\documentclass[letterpaper,12pt]{article}
\title{A bijective proof and generalization of Siladi\'c's Theorem}
\author{Isaac KONAN \\ IRIF, Universit\'e Paris Diderot, Paris,  75013, France \\ konan@irif.fr} 
\date{}
\usepackage{amssymb}
\usepackage{amsmath}
\usepackage{amsthm}
\usepackage[top=1in,bottom=1in,left=1in,right=1in,includehead]{geometry}
\usepackage{enumitem}
\usepackage{xcolor}
\usepackage{graphicx}
\usepackage{pgfarrows,pgfnodes}
\usepackage{url}
\usepackage{textcomp}
\usepackage{ytableau}

\definecolor{foge}{rgb}{0.1, 0.6, 0.1}

\newcommand{\So}{\textbf{Step 1 }}
\newcommand{\Soo}{\textbf{Step 1}}
\newcommand{\St}{\textbf{Step 2 }}
\newcommand{\Stt}{\textbf{Step 2}}
\newcommand{\Thm}[1]{\textbf{Theorem \ref{#1}}}
\newcommand{\Lem}[1]{\textbf{Lemma \ref{#1}}}
\newcommand{\Prp}[1]{\textbf{Proposition \ref{#1}}}
\newcommand{\Sct}[1]{Section \ref{#1}}
\newcommand{\Od}{\mathcal{O}}
\newcommand{\C}{\mathcal{C}}
\newcommand{\D}{\mathcal{D}}

\newcommand{\Cr}{\textit{Chasles' relation }}
\newcommand{\Crr}{\textit{Chasles' relation}}

\newcommand{\E}{\mathcal{E}}
\newcommand{\ab}{\textcolor{blue}{a}}

\newcommand{\ba}{\textcolor{foge}{b}}
\newcommand{\aba}{\textcolor{blue}{a^2}}
\newcommand{\bab}{\textcolor{foge}{b^2}}
\newcommand{\abab}{\ab\ba}
\newcommand{\baba}{\ba\ab}
\newcommand{\la}{\lambda}
\newcommand{\sss}{\{1,\ldots,s\}}
\newcommand{\ssss}{\{1,\ldots,s-1\}}

\numberwithin{equation}{section}
\newtheorem{theo}{Theorem}[section]
\newtheorem{prop}[theo]{Proposition}
\newtheorem{lem}[theo]{Lemma}
\newtheorem{cor}[theo]{Corollary}
\newtheorem{rem}[theo]{Remark}

\begin{document}
\maketitle
\ytableausetup{centertableaux,boxsize=0.31cm}
\begin{abstract}
 In a recent paper, Dousse introduced a refinement of Siladi\'c's theorem on partitions, where parts occur in two primary and three secondary colors. Her proof used the method of weighted words and $q$-difference equations. 
The purpose of this paper is to give a bijective proof of a generalization of Dousse's theorem  from two primary colors to an arbitrary number of primary colors.
\end{abstract}
\tableofcontents
\newpage
\section{Introduction}\label{intro}
In this paper, we denote by
$\lambda_1+\cdots+\lambda_s$ a partition of a non-negative integer  $n$. For any complex numbers $x,q\text{ with }\,\vert q\vert<1$, and any non-negative integer $n$, we define
\[(x;q)_n = \prod_{k=0}^{n-1} (1-xq^k)\,\,,\]
with the convention $(x;q)_0=1$, and 
\[(x;q)_\infty = \prod_{k=0}^\infty (1-xq^k)\,\cdot\]
Recall the Rogers-Ramanujan identities \cite{RR19}, which state that for $a\in \lbrace 0,1\rbrace$,
\begin{equation}
\sum_{n\geq 0} \frac{q^{n(n+a)}}{(q;q)_n} = \frac{1}{(q^{1+a};q^5)_\infty(q^{4-a};q^5)_\infty}\,\cdot
\end{equation}
These identities give an equality between the cardinalities of two sets of partitions : the set of partitions of $n$ with parts differing by at least two and greater than $a$, and  the set of partitions of $n$ with parts congruent to $1+a,4-a \mod 5$.
In the spirit of these identities, a $q$-series or combinatorial identity is said to be of \textit{Rogers-Ramanujan type} if it links some sets of partitions with certain difference conditions to others with certain congruence conditions. 
Another well-known example is Schur's partition theorem \cite{Sc26}, which states that the number of partitions of $n$ into parts congruent to $\pm 1 \mod 6$ is equal to the number of partitions of $n$ where parts differ by at least three and multiples of three differ by at least six.\\\\
A rich source of such identities is the representation theory of Lie algebras. This has its origins in work of Lepowsky and Wilson \cite{LW84}, who proved the Rogers-Ramanujan identities
by using representations of the affine Lie algebra $\mathfrak{sl}(2,\mathbb{C})^\sim$. 
Subsequently, Capparelli \cite{C93}, Meurman-Primc \cite{MP87,MP99} and others examined
related standard modules and affine Lie algebras and found many new
Rogers-Ramanujan type identities.
\\Our motivation in this paper is one such identity given by Siladi\'c \cite{Si02} in his study of representations of the twisted affine Lie algebra $A_2^{(2)}$.
\begin{theo}[Siladi\'c, rephrased by Dousse]\label{th1}
The number of partitions $\lambda_1+\cdots+\lambda_s$ of an integer $n$ into distinct odd parts 
is equal to the number of partitions of $n$ into parts different from $2$ such that 
$\lambda_i -\lambda_{i+1}\geq 5$ and 
\[\lambda_i -\lambda_{i+1}  = 5 \Rightarrow \, \lambda_i  \equiv 1,4 \mod 8\,,\]
\[\lambda_i -\lambda_{i+1}  = 6 \Rightarrow \, \lambda_i \equiv 1,3,5,7\mod 8\,,\]
\[\lambda_i -\lambda_{i+1}  = 7 \Rightarrow \, \lambda_i \equiv 0,1,3,4,6,7 \mod 8\,,\]
\[\lambda_i -\lambda_{i+1}  = 8 \Rightarrow \, \lambda_i \equiv 0,1,3,4,5,7\mod 8\,\,\cdot\]
\end{theo}
For example, for $n=16$, the partitions into distinct odd parts are
\[15+1, 13+3, 11+5, 9+7\text{ and }7+5+3+1\,,\]
while the partitions of the second kind are
\[15+1, 13+3, 11+5, 16\text{ and } 12+4\,\cdot \]
Siladi\'c's theorem has recently been refined by Dousse \cite{Dousse16}.
She was inspired by the original method of weighted words, first introduced by Alladi and Gordon \cite{AG93} to generalize Schur's partition theorem.
Her framework is as follows: we consider parts colored by two primary colors $a,b$ and three secondary colors $a^2,b^2,ab$, with the colored parts ordered by 
\begin{equation}\label{ord}
1_{ab}<_c1_a<_c1_{b^2}<_c1_b<_c2_{ab}<_c2_a<_c3_{a^2}<_c2_b<_c3_{ab}<_c3_a<_c3_{b^2}<_c3_b<_c\cdots\,\cdot
\end{equation}
Note that only odd parts can be colored by $a^2,b^2$. The dilation
\begin{equation}\label{dila}
q\rightarrow q^4\,,
a\rightarrow aq^{-3}\,,
b\rightarrow bq^{-1}\,,
\end{equation}
leads to the natural order 
\begin{equation}
0_{ab}<1_a<2_{b^2}<3_b<4_{ab}<5_a<6_{a^2}<7_b<8_{ab}<9_a<10_{b^2}<11_b<\cdots\,\cdot
\end{equation}
We then impose the minimal differences according the following table
\begin{equation}\label{A}
\begin{array}{|c|cccccccc|}
\hline
_{\lambda_i}\setminus^{\lambda_{i+1}}&a^2_{odd}&a_{odd}&a_{even}&b^2_{odd}&b_{odd}&b_{even}&ab_{odd}&ab_{even}\\
\hline
a^2_{odd}&4&4&3&4&4&3&4&3\\
a_{odd}&2&2&3&2&2&3&2&1\\
a_{even}&3&3&2&3&3&2&3&2\\
b^2_{odd}&2&2&3&4&4&3&2&3\\
b_{odd}&2&2&1&2&2&3&2&1\\
b_{even}&1&1&2&3&3&2&1&2\\
ab_{odd}&2&2&3&4&4&3&2&3\\
ab_{even}&3&3&2&3&3&2&3&2\\
\hline
\end{array}\,,
\end{equation}
which can be reduced to the table :  
\begin{equation}\label{Aa}
\begin{array}{|c|cccccc|}
\hline
_{\lambda_i}\setminus^{\lambda_{i+1}}&a^2_{odd}&a&b^2_{odd}&b&ab_{odd}&ab_{even}\\
\hline
a^2_{odd}&4&3&4&3&4&3\\
a&2&2&2&2&2&1\\
b^2_{odd}&2&2&4&3&2&3\\
b&1&1&2&2&1&1\\
ab_{odd}&2&2&4&3&2&3\\
ab_{even}&3&2&3&2&3&2\\
\hline
\end{array}\,\cdot
\end{equation}
One can check that these minimal differences define a partial strict order on the set of parts colored by primary and secondary colors. We denote it by $\gg_c$, so that stating that $\lambda_i - \lambda_{i+1}$ respects the minimal difference condition is equivalent to saying that $\lambda_i\gg_c\lambda_{i+1}$.
With this coloring, Dousse refined the Siladi\'c theorem  as follows:
\begin{theo}[Dousse]\label{th2}Let $(u,v,n) \in \mathbb{N}^3$.
Denote by  $\mathcal{D}(u,v,n)$ the set of all the partitions of $n$, such that no part is equal to $1_{ab},1_{a^2}$ or $1_{b^2}$, with the difference between two consecutive parts following the minimal conditions in \eqref{A}, and with $u$ equal to the number of parts with color $a$ or $ab$ plus twice the number of parts colored by $a^2$, and $v$ equal to the number of parts with color $b$ or $ab$ plus twice the number of parts colored by $b^2$.
Denote by $\mathcal{C}(u,v,n)$ the set of all the partitions of $n$ with $u$ distinct parts colored by $a$ and $v$ distinct parts colored by $b$. We then have $\sharp \mathcal{D}(u,v,n) = \sharp \mathcal{C}(u,v,n)$.
\end{theo}
In terms of $q$-series, we have the equation
\begin{equation}\label{eq111}
\sum_{u,v,n\geq 0}  \sharp \mathcal{D}(u,v,n) a^u b^v q^n = \sum_{u,v,n\geq 0}  \sharp \mathcal{C}(u,v,n) a^u b^v q^n =(-aq;q)_{\infty} (-bq;q)_{\infty}\,\cdot
\end{equation}
Dilating \eqref{A} by \eqref{dila} gives exactly the minimal difference conditions in Siladi\'c's theorem and \eqref{eq111} becomes the generating function for partitions into distinct odd parts, so that \Thm{th1} is a corollary of \Thm{th2}.
By restricting the set of primary colors to $\{a\}$, and applying the transformation $(q,a)\mapsto (q^2,q^{-1})$, one can derive the following analogue of Schur's theorem.
\begin{theo}
Let $n$ be a positive integer. The number of partitions of $n$ into distinct odd parts is equal to the number of partitions of $n$ into distinct positive integers, odd or multiples of $4$, such that 
two consecutive parts differ by at least $4$, and the two multiples of $4$ differ by at least $8$.
\end{theo}
Our purpose here is to build a bijection for a generalization of Dousse's theorem to an arbitrary  number of primary colors.
We consider a set of $m$ primary colors $a_1<\cdots<a_m$. And we order the parts colored by primary colors in the usual way, first according to size and then according to color (see \eqref{odd}). 
We also set $m^2$ secondary colors $a_ia_j$ with $i,j\in \{1,\dots,m\}$, in such a way that  $a_ia_j$ only colors parts with the same parity as $\chi(a_i\leq a_j)$, where
$\chi(A)=1$ if $A$ is true and $\chi(A)=0$ if not (see \eqref{div}).
\\We then extend the partial order $\gg_c$ to parts colored with primary and secondary colors, which corresponds to minimal difference conditions between the parts (see \Sct{cdtdif} and \Sct{genpr}).  This leads to the following theorem.
\begin{theo}\label{th3}
Let $\mathcal{C}(u_1,\dots,u_m,n)$ denote the set of partitions of $n$ with $u_k$ distinct parts with color $a_k$.
Let $\mathcal{D}(u_1,\dots,u_m,n)$ denote the set of partitions of $n$ such that parts are ordered by $\gg_c$, with no part equal to $1_{a_ia_j}$, and with $u_i$ equal to the number of parts colored by $a_i$, $a_ia_j$ or $a_ja_i$ with $i\neq j$, plus twice the number of parts colored by $a_i^2$. We then have
\begin{equation}
\sharp \mathcal{C}(u_1,\dots,u_m,n) = \sharp \mathcal{D}(u_1,\dots,u_m,n)\,\cdot
\end{equation}
\end{theo}
In terms of $q$-series, we have the equation
\begin{equation}
\begin{array}{rcl}
 \displaystyle{\sum_{u_1,\dots,u_m,n\geq 0}} \sharp \mathcal{D}(u_1,\dots,u_m,n) a_1^{u_1}\cdots a_m^{u_m} q^n &=& \displaystyle{\sum_{u_1,\dots,u_m,n\geq 0}}   \sharp \mathcal{C}(u_1,\dots,u_m,n) a_1^{u_1}\cdots a_m^{u_m} q^n\\
 \\
 &=& (-a_1q;q)_{\infty}\cdots (-a_mq;q)_{\infty}\,\cdot
\end{array}
\end{equation}
A complete version of the above theorem is given in \Thm{th4}.
This result may be compared with work of Corteel and Lovejoy \cite{CL06} who gave interpretations of the same infinite products above but using $2^{m} -1$ colors instead of $m^2+m$ colors as we do here.
As an example, we choose $m=3$ and use $a,b,d$ instead of $a_1,a_2,a_3$.
The table which sums up the minimal differences is
\[\begin{footnotesize}
   \begin{array}{|c||c|c|c||c|c|c|c|c|c||c|c|c|}
\hline
_{\lambda_i}\setminus^{\lambda_{i+1}}&a&b&d&a^2&ab&ad&b^2&bd&d^2&ba&da&db\\
\hline\hline
a&2&2&2&2&2&2&2&2&2&1&1&1\\
\hline
b&1&2&2&1&1&1&2&2&2&1&1&1\\
\hline
d&1&1&2&1&1&1&1&1&2&0&1&1\\
\hline\hline
a^2&3&3&3&4&4&4&4&4&4&3&3&3\\
\hline
ab&2&3&3&2&2&2&4&4&4&3&3&3\\
\hline
b^2&2&3&3&2&2&2&4&4&4&3&3&3\\
\hline
ad&2&2&3&2&2&2&2&2&4&1&3&3\\
\hline
bd&2&2&3&2&2&2&2&2&4&1&3&3\\
\hline
d^2&2&2&3&2&2&2&2&2&4&1&3&3\\
\hline\hline
ba&2&2&2&3&3&3&3&3&3&2&2&2\\
\hline
da&2&2&2&3&3&3&3&3&3&2&2&2\\
\hline
db&1&2&2&1&1&1&3&3&3&2&2&2\\
\hline
\end{array}
\,\,\cdot
  \end{footnotesize}
\]
If we take the dilation 
\[
\left\lbrace\begin{array}{l}
q\mapsto q^{10}\\
a \mapsto a q^{-6}\\
b\mapsto b q^{-4}\\
d \mapsto d q^{-1}\\
\end{array}\right.\,\,,
\]
and use the order
\begin{eqnarray}
 1_{ab}<_c 1_{b^2}<_c 1_{ad}<_c 1_{a}<_c 1_{bd}<_c 1_{b}<_c 1_{d^2}<_c 1_{d}<_c\nonumber\\
2_{ba}<_c 2_{da}<_c 2_{a}<_c 2_{db}<_c 2_{b}<_c3_{a^2} <_c2_{d}<_c 3_{ab}<_c\cdots 
\end{eqnarray}
induced by the natural ordering
\begin{eqnarray}
 0_{ab}< 2_{b^2}< 3_{ad}< 4_{a}< 5_{bd}< 6_{b}< 8_{d^2}< 9_{d}<10_{ba}<\nonumber\\
 13_{da}< 14_{a}< 15_{db}< 16_{b}< 18_{a^2} < 19_{d}< 20_{ab}<\cdots \,,
\end{eqnarray}
we have the following corollary in the spirit of Siladi\'c's theorem :
\begin{cor} Let $u,v,w,n$ be non-negative integers. 
Let $A(u,v,w,n)$ denote the number of partitions of $n$ with respectively $u,v,w$ parts congruent to $4,6,9 \mod 10$.
Let $B(u,v,w,n)$ denote the number of partitions  $\lambda_1+\cdots+\lambda_s$ of $n$, with
\begin{itemize}
 \item no part equal to $2,3,5,8$ or congruent to $1,7,11,12,17 \mod 20$,
 \item $u$ equal to the number of parts congruent to $0,3,4\mod 10$ plus twice the number of parts congruent to $18 \mod 20$,
 \item $v$ equal to the number of parts congruent to $0,5,6\mod 10$ plus twice the number of parts congruent to $2 \mod 20$,
 \item $w$ equal to the number of parts congruent to $3,5,9\mod 10$ plus twice the number of parts congruent to $8 \mod 20$,
 \item two consecutive parts differing by at least $9$ with the additional conditions \\for $9\leq \lambda_i-\lambda_{i+1}\leq 20$ according to the table below:
 \begin{equation}
  \begin{array}{|c|c||c|c|}
  \hline
   \lambda_i - \lambda_{i+1}& \lambda_i \mod 20&\lambda_i - \lambda_{i+1}& \lambda_i \mod 20\\
   \hline
   9& 4,19&15&4,5,9,10,14,15,19\\
   10& \emptyset&16&0,4,6,9,10,15,16,19\\
   11& 4,6,10,15&17&0,3,6,10,13,15,16,19\\
   12&6,15,16&18&2,3,4,6,8,13,14,16\\
   13&3,6,9,16,19&19&2,3,4,5,9,13,14,15,18,19\\
   14& 4,9,10,13,19&20&0,3,4,5,6,9,10,13,14,15,16,19\\
   \hline
  \end{array} \cdot
 \end{equation}
\end{itemize}
Then $A(u,v,w,n)=B(u,v,w,n)$.
\end{cor}
As an example, for $n=48$, for $(u,v,w)\notin\{(2,0,0),(3,1,0),(0,3,0),(1,1,2),(0,0,2)\}$, 
$A(u,v,w,48)=0$, and 
\[
\begin{array}{|c||c|c|}
\hline
(u,v,w)&\text{type }A&\text{type }B\\
\hline\hline
(2,0,0)&44+4&44+4\\
\hline
(3,1,0)&24+14+6+4&38+10\\
\hline
(0,3,0)&26+16+6&42+6\\
\hline
(1,1,2)&19+16+9+4,19+14+9+6,29+9+6+4&35+13,33+15,29+15+4\\
\hline
(0,0,2)&39+9,29+19&39+9,48\\
\hline
\end{array}\,\cdot
\]
The remainder of the paper is organized as follows. In the next section, we discuss the existence of a new color $ba$ different from $ab$, that will lead to an enumeration of explicit relations
for the minimal difference conditions in \eqref{Aa}. We also indicate how to generalize this to an arbitrary number of primary colors. In \Sct{bij}, we will build our bijection, and finally in \Sct{prf}, prove its well-definedness.
\section{Preliminaries}\label{prel}
In this section, we examine the difference conditions in \eqref{Aa} and extend them to an arbitrary number of primary colors.
\subsection{A new color $\baba\neq \abab$}\label{ncol}
First, we set  $\ab<\ba$ for the primary colors. Then for any $(k,l,p,q) \in \mathbb{N}^2\times\lbrace \ab,\ba\rbrace^2$, the usual order $>_c$ can be defined by the equivalence
\begin{equation} \label{odd}
k_p>_c l_q \Leftrightarrow k-l \geq \chi(p\leq q)\,\cdot
\end{equation}
At the same time, by observing \eqref{Aa}, the resulting table for primary colors is
\[\begin{array}{|c||c|c|}
   \hline
   _{\lambda_i}\setminus^{\lambda_{i+1}}&\ab&\ba\\
   \hline
   \hline
   \ab&2&2\\
   \hline
   \ba&1&2\\
   \hline
  \end{array}\,,
\]
so that the relation
\begin{equation} \label{Aodd}
k_p\gg_c l_q \Leftrightarrow k-l \geq 1+\chi(p\leq q)\,
\end{equation}
holds for any $(k,l,p,q) \in  {\mathbb{N}^*}^2\times\lbrace \ab,\ba\rbrace^2$.
We set $\delta_{pq}=\chi(p\leq q)$ for notational convenience.
\\\\One can observe that the parts colored by $\aba,\bab,\abab$ can be uniquely divided into two parts $k_p,l_q$ colored by primary colors such that $k_p>_c l_q$ and  $k_p\not \gg_c l_q$, i.e $k-l = \delta_{pq}$. Specifically, we have 
\[
\begin{array}{c}
(2k)_{\abab} = k_{\ba}+k_{\ab}\\
(2k+1)_{\abab} = (k+1)_{\ab}+ k_{\ba}\\
(2k+1)_{\aba} = (k+1)_{\ab}+ k_{\ab}\\
(2k+1)_{\bab} = (k+1)_{\ba}+ k_{\ba}
\end{array}\,\cdot
\]
It is then convenient to set another color $\baba$ so that for any $p,q\in \lbrace\ab,\ba\rbrace$, $pq$ only colors parts with the same parity as $\delta_{pq}$. Futhermore, the following equality holds :
\begin{equation}\label{div}
(2k+\delta_{pq})_{pq} = (k+\delta_{pq})_p+k_q\,\cdot
\end{equation}
For example 
\[\begin{ytableau}
\none $8_{\baba}$&\none $=$
&\none $\baba$&\none &*(foge)  &*(foge) &*(foge)  &*(foge)  &*(blue) &*(blue) &*(blue) &*(blue)&\none $=$&\none$\ba$&\none & *(foge) &*(foge)& *(foge) &*(foge)&\none $+$ &\none $\ab$&\none &*(blue)&*(blue)&*(blue)&*(blue) 
&\none $=$&\none$4_{\ba}$&\none $+$ &\none $4_{\ab}$
\end{ytableau}\,\cdot\]
This means that $\abab\neq \baba$, since $\abab$ colors only odd parts and $\baba$ only even parts. \\The above notations allow us to introduce two sets of parts :
\begin{itemize}
 \item $\mathcal{O} = \mathbb{Z}\times\lbrace \ab,\ba\rbrace$ for parts colored by primary colors, in such a way that $k_p$ is represented by $(k,p)$.
 \item $\mathcal{E} = \mathbb{Z}\times\lbrace \ab,\ba\rbrace^2$,
so that the part $(2k+\delta_{pq})_{pq}$ is uniquely represented by $(k,p,q)$. For example, we have 
\begin{equation}
\begin{array}{ccccc}
\begin{ytableau}\none $\baba$&\none &*(foge)  &*(foge) &*(foge)  &*(foge) &*(blue) &*(blue)&*(blue) &*(blue)\end{ytableau}
&\leftrightarrow&(8)_{\baba}&\leftrightarrow&(4,\ba,\ab)\\
\begin{ytableau}\none $\abab$ &\none&*(blue) &*(blue)&*(blue) &*(blue)&*(blue)&*(foge)  &*(foge) &*(foge)  &*(foge) \end{ytableau}
&\leftrightarrow&(9)_{\abab} &\leftrightarrow&(4,\ab,\ba)\\
\begin{ytableau}\none $\aba$ &\none&*(blue) &*(blue)&*(blue) &*(blue)&*(blue)&*(blue) &*(blue)&*(blue)&*(blue) &*(blue)&*(blue) \end{ytableau}
&\leftrightarrow&(11)_{\aba}& \leftrightarrow&(5,\ab,\ab)\\
\begin{ytableau}\none $\bab$&\none &*(foge) &*(foge)  &*(foge)&*(foge) &*(foge) &*(foge)  &*(foge) \end{ytableau}
&\leftrightarrow&(7)_{\bab} &\leftrightarrow&(3,\ba,\ba)
\end{array}\,\cdot
\end{equation}
\end{itemize}
We remark that we take these sets on $\mathbb{Z}$, while the size of a part is a positive integer. This extension to the integers is meant for easing our bijection in its construction.
\\It is reasonable to set $\gamma,\mu$ to be the functions which map a part of $\mathcal{E}$ to the unique parts in $\mathcal{O}$ as 
\begin{equation}\label{half}
\gamma(k,p,q) = (k+\delta_{pq},p)\,\,,\quad \mu(k,p,q)=(k,q)\,\cdot
\end{equation}
We call $\gamma(k,p,q)$ and $\mu(k,p,q)$ respectively the \textbf{upper} and the \textbf{lower} halves of $(k,p,q)$.
As an example, the part $40_{\abab}$ considered by Dousse will be in fact the part $40_{\baba}$, which we denote $(20,\ba,\ab)$ and which is the sum of the unique parts $20_{\ba}=(20,\ba)$ and $20_{\ab}=(20,\ab)$ respectively as its upper half $\gamma$ and its lower half $\mu$.
\\With this notation, a part $(k,p)\in\mathcal{O}$ has an actual size $k$, while a part $(l,q,r)\in\mathcal{E}$ has $2l+\delta_{qr}$ as its actual size.
\subsection{Explicit relations for the minimal difference conditions}\label{cdtdif}
Since the order $\gg_c$ is defined by the minimal differences between parts in \eqref{Aa}, we can extend it to 
\begin{equation}
 \label{go}
 \mathcal{P} = \mathcal{O}\cup \mathcal{E}\,\cdot
\end{equation}
We just saw in the previous section the necessary and sufficient condition \eqref{Aodd} to have the minimal difference between two parts colored by primary colors. We state it as a lemma.
\begin{lem}
For any $(k,p),(l,q)\in \mathcal{O}^2$, we have 
\begin{equation}\label{oo}
(k,p)\gg_c(l,q)\,\Leftrightarrow \, k-l \geq 1+\delta_{pq}\,\,\cdot
\end{equation}
\end{lem}
Now we are going to give analogous conditions for any pair of parts in $\mathcal{O}\times\mathcal{E},\mathcal{E}\times\mathcal{E},\mathcal{E}\times\mathcal{O}$, by giving some explicit expressions of the minimal difference conditions given in \eqref{Aa} according to the colors involved.
\begin{lem}
For any $(k,p),(l,q,r)\in \mathcal{O}\times\mathcal{E}$, we have 
\begin{equation}\label{oe}
(k,p)\gg_c(l,q,r)\,\Leftrightarrow k-(2l+\delta_{qr})\geq \delta_{pq}+\delta_{qr}\,\Leftrightarrow (k,p)>_c(2(l+\delta_{qr}),q)\,\cdot
\end{equation}
\end{lem}
\begin{proof}[Proof]
 The subtable resulting from \eqref{Aa}  and corresponding to these differences is 
 \[\begin{array}{|c||c|c|c|c|}
   \hline
   _{\lambda_i}\setminus^{\lambda_{i+1}}&\aba&\bab&\abab&\baba\\
   \hline
   \hline
   \ab&2&2&2&1\\
   \hline
   \ba&1&2&1&1\\
   \hline
  \end{array}\,,
\]
and it is exactly equivalent to the expression $\delta_{pq}+\delta_{qr}$.
\end{proof}
We prove the other lemmas by using the corresponding subtables of \eqref{Aa}.
\begin{lem}
For any $(k,p,q),(l,r)\in \mathcal{E}\times\mathcal{O}$, we have 
\begin{equation}\label{eo}
(k,p,q)\gg_c(l,r)\,\Leftrightarrow (2k+\delta_{pq})-l\geq 1+ \delta_{pq}+\delta_{qr} \Leftrightarrow \, (2k,q)\gg_c (l,r)\,\cdot
\end{equation}
\end{lem}
\begin{lem}
For any $(k,p,q),(l,r,s)\in \mathcal{E}^2$, we have 
\begin{equation}\label{ee}
(k,p,q)\gg_c(l,r,s)\,\Leftrightarrow (2k+\delta_{pq})-(2l+\delta_{rs})\geq \delta_{pq}+2\delta_{qr}+\delta_{rs}\,\cdot
\end{equation}
Furthermore, the last equality is equivalent to
\begin{equation}\label{gm}
(k,p,q)\gg_c(l,r,s)\,\Leftrightarrow k-(l+\delta_{rs})\geq \delta_{qr}\,\Leftrightarrow\,\mu(k,p,q)>_c\gamma(l,r,s)\,\cdot
\end{equation}
\end{lem}
Condition \eqref{gm} is the most important in our construction. This comes from the fact that comparing two parts in $\mathcal{E}$ in terms of $\gg_c$  is the same as comparing the lower half of the first part and the upper half of the second part using $>_c$.
\subsection{Generalization to an arbitrary number of primary colors}\label{genpr}
The most important fact in our analysis of the colored parts  in the previous subsection is the order between primary colors and not their number. In fact, this extends immediately to a set of primary colors $\{A_1,\ldots,a_m\}$ as follows. After ordering $a_1<\dots<a_m$, we set
\begin{equation}\label{generalo}
 \mathcal{O} = \mathbb{Z}\times \lbrace a_1,\dots ,a_m\rbrace\,,\, \mathcal{O}' = \mathbb{Z}_{>0} \lbrace a_1,\dots ,a_m\rbrace
\end{equation}
for the parts with primary colors and 
\begin{equation}\label{generale}
 \mathcal{E} = \mathbb{Z}\times \lbrace a_1,\dots ,a_m\rbrace^2\,,\, \mathcal{E}' = \mathbb{Z}_{>0} \lbrace a_1,\dots ,a_m\rbrace^2
\end{equation}
for the parts with secondary colors, whose size and color are defined exactly by \eqref{div}, and upper and whose lower halves are defined by \eqref{half}. We can then define the usual order $>_c$ described in \eqref{odd} and use the lemmas of \Sct{cdtdif} as definitions of $\gg_c$.
We now explicitly state the generalization of Siladi\'c's theorem
\begin{theo}\label{th4}
Let $\mathcal{C}(u_1,\dots,u_m,n)$ denote the set of partitions of $n$ with $u_k$ distinct primary parts in parts $\mathcal{O}'$ with color $a_k$.
Let $\mathcal{D}(u_1,\dots,u_m,n)$ denote the set of partitions of $n$ with parts in $\mathcal{O}'\cup \mathcal{E}'$ satisfying the minimal difference condition defined by $\gg_c$ in \eqref{oo}, \eqref{oe}, \eqref{eo} and \eqref{ee},
and with $u_i$ equal to the number of parts colored by $a_i$, $a_ia_j$ or $a_ja_i$ with $i\neq j$, plus twice the number of parts colored by $a_i^2$. We then have
\begin{equation}
\sharp \mathcal{C}(u_1,\dots,u_m,n) = \sharp \mathcal{D}(u_1,\dots,u_m,n)\,\cdot
\end{equation}
\end{theo}
\section{How do we build the bijection?}\label{bij}
We build our bijection for \Thm{th4} in the spirit of the bijective proof of the partition theorem of K. Alladi \cite{Alladi99} given by Padmavathamma, R. Raghavendra and B. M. Chandrashekara \cite{PRC04}.
The idea was introduced by Bressoud \cite{BR80}  in his bijective proof of Schur's theorem.
\subsection{The bijection's key operation $\Lambda$}\label{kop}
Let us define the operation $\Lambda$ as 
\begin{equation}
\begin{array}{l c r c l}
\Lambda&:&\mathcal{O}\times\mathcal{E}&\longrightarrow&\mathcal{E}\times\mathcal{O}\\
&&(k,p),(l,q,r)&\longmapsto& (l+\delta_{qr},p,q),(k-\delta_{pq}-\delta_{qr},r)\\
\end{array}\,\,\cdot
\end{equation}
The function $\Lambda$ is invertible with 
\begin{equation}
\begin{array}{l c r c l}
\Lambda^{-1}&:&\mathcal{E}\times\mathcal{O}&\longrightarrow&\mathcal{O}\times\mathcal{E}\\
&&(l',p,q),(k',r)&\longmapsto& (k'+\delta_{pq}+\delta_{qr},p),(l'-\delta_{qr},q,r)\\
\end{array}\,\,\cdot
\end{equation}
For two colors $\ab,\ba$, we explicitly have the following table for $\Lambda$:
\begin{footnotesize}
\[
\begin{array}{|c||c|c|c|c|}
\hline
_{(k,p)}\times^{(l,q,r)}&(l,\ab,\ab)&(l,\ab,\ba)&(l,\ba,\ab)&(l,\ba,\ba)\\
\hline\hline
(k,\ab)&(l+1,\ab,\ab),(k-2,\ab)&(l+1,\ab,\ab),(k-2,\ba)&(l,\ab,\ba),(k-1,\ab)&(l+1,\ab,\ba),(k-2,\ba)\\
\hline
(k,\ba)&(l+1,\ba,\ab),(k-1,\ab)&(l+1,\ba,\ab),(k-1,\ba)&(l,\ba,\ba),(k-1,\ab)&(l+1,\ba,\ba),(k-2,\ba)\\
\hline
\end{array}
\]
\end{footnotesize}
and with the actual sizes 
\begin{equation}
\begin{small}
\begin{array}{|c||c|c|c|c|}
\hline
_{k_p}\times^{(2l+\delta_{qr})_{qr}}&(2l+1)_{\aba}&(2l+1)_{\abab}&(2l)_{\baba}&(2l+1)_{\bab}\\
\hline\hline
k_{\ab}&(2l+3)_{\aba},(k-2)_{\ab}&(2l+3)_{\aba},(k-2)_{\ba}&(2l+1)_{\abab},(k-1)_{\ab}&(2l+3)_{\abab},(k-2)_{\ba}\\
\hline
k_{\ba}&(2l+2)_{\baba},(k-1)_{\ab}&(2l+2)_{\baba},(k-1)_{\ba}&(2l+1)_{\bab},(k-1)_{\ab}&(2l+3)_{\bab},(k-2)_{\ba}\\
\hline
\end{array}\,\cdot
\end{small}
\end{equation}
By considering the upper and lower halves, we have the following transformation for $\Lambda$
\begin{equation}\label{mov}
 \begin{array}{ccc}
(k,p)&(l+\delta_{qr},q)&(l,r)\\
\rightarrow\rightarrow&\leftarrow&\leftarrow\\
(l+\delta_{pq}+\delta_{qr},p)&(l+\delta_{qr},q)&(k-\delta_{pq}-\delta_{qr},r)
\end{array}\,\,\,,
\end{equation}
and a similar  transformation for $\Lambda^{-1}$
\begin{equation}\label{mov1}
 \begin{array}{ccc}
(l'+\delta_{pq},p)&(l',q)&(k',r)\\
\rightarrow&\rightarrow&\leftarrow\leftarrow\\
(k'+\delta_{pq}+\delta_{qr},p)&(l',q)&(l'-\delta_{qr},r)
\end{array}\,\,\,\cdot
\end{equation}
Observe that the sum of the sizes is conserved by $\Lambda$ and $\Lambda^{-1}$, and the same goes for the sequence of primary colors of the parts. Two more properties of the operation $\Lambda$ are important in the construction of our bijection.
\begin{prop}\label{pr1}
 For any $(k,p),(l,q,r)\in \mathcal{O}\times\mathcal{E}$,
\begin{equation}
(k,p)\not\gg_c(l,q,r) \Leftrightarrow (l+\delta_{qr},p,q)\gg_c (k-\delta_{pq}-\delta_{qr},r)\,\cdot
\end{equation}
This means  applying $\Lambda$ turns a pair of primary-colored and secondary-colored parts not satisfying  $\gg_c$ into a new pair that does.
\end{prop}
\begin{proof}[Proof]
 By \eqref{oe}, the left side is equivalent to 
 \[k-(2l+\delta_{qr})<\delta_{pq}+\delta_{qr} \Leftrightarrow k -2l \leq -1+\delta_{pq}+2\delta_{qr}\]
 while by \eqref{eo}, the right side means
 \[(2l+2\delta_{qr}+\delta_{pq})-(k-\delta_{pq}-\delta_{qr})\geq 1+\delta_{pq}+\delta_{qr} \Leftrightarrow 2l-k \geq 1-\delta_{pq}-2\delta_{qr}\,\cdot\]
\end{proof}
\begin{prop}\label{pr2}
 For any $(k,p,q),(l,r)\in \mathcal{E}\times\mathcal{O}$,
\begin{equation}
\mu(k,p,q)\not >_c (l,r) \Leftrightarrow (l+\delta_{pq}+\delta_{qr},p) \gg_c \gamma(k-\delta_{qr},q,r)\,\cdot
\end{equation}
\end{prop}
\begin{proof}[Proof]
 The left side is equivalent to \[(k,q) \not>_c (l,r) \Leftrightarrow k-l<\delta_{qr} \Leftrightarrow k-l \leq -1+\delta_{qr}\]
 and for the right side, we have 
 \[(l+\delta_{pq}+\delta_{qr},p)\gg_c (k,q)\Leftrightarrow(l+\delta_{pq}+\delta_{qr})-k\geq 1+\delta_{pq} \Leftrightarrow l-k \geq 1 - \delta_{qr}\,\cdot\]
\end{proof}
\subsection{Bijective maps}
Let us define $\mathcal{P}' = \mathcal{O}'\cup \mathcal{E}'$.
\begin{itemize}
 \item Denote by $\mathcal{C}$ the set of partitions with parts in primary colors, i.e with parts in $\mathcal{O}'$. We can then view $\mathcal{C}$ as the set of all finite decreasing chains of the totally ordered set $(\mathcal{O}',>_c)$.
 \item Let $\mathcal{D}$  denote the set of partitions with parts in $\mathcal{P'}$
such that the colored parts are ordered by $\gg_c$. Here again, $\mathcal{D}$ is the set of all finite decreasing chains of the poset $(\mathcal{P}',\gg_c)$. Observe that a part $(k,p,q)\in\mathcal{E}'$ has an actual size $2k+\delta_{pq}\geq 2$, so that there is no secondary part of size $1$.
\end{itemize}
We shall define a suitable mapping $\Phi$ from $\mathcal{C}$ to $\mathcal{D}$ and suitable mapping $\Psi$ from $\mathcal{D}$ to $\mathcal{C}$.
\subsubsection{How to compute $\Phi:\mathcal{C}\rightarrow\mathcal{D}$}\label{phii}
Let us take any $\lambda = \lambda_1+\cdots+\lambda_s$ in $\mathcal{C}$, with $\lambda_1>_c\cdots >_c \lambda_s$. We then have $\lambda_i\in\mathcal{O}'$ for any $i\in\lbrace 1,\dots,s\rbrace$.
As an example, we take 
\begin{equation}\label{exp}
\lambda = 24_{\ab} + 17_{\ba}+ 11_{\ba}+10_{\ab}+9_{\ba}+8_{\ba}+6_{\ab}+5_{\ab}+4_{\ba}+4_{\ab}\,\,,
\end{equation}
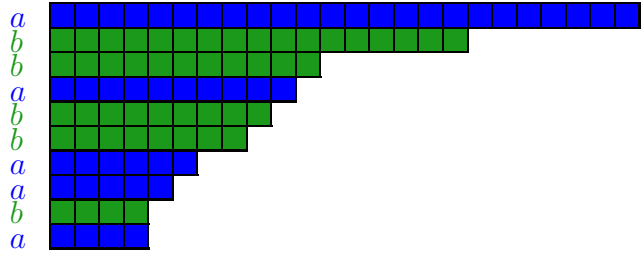
\begin{center}
\begin{ytableau}
\none $\ab$&\none&*(blue)&*(blue)&*(blue)&*(blue)&*(blue)&*(blue)&*(blue)&*(blue)&*(blue)&*(blue)&*(blue)&*(blue)&*(blue)&*(blue)&*(blue)&*(blue)&*(blue)&*(blue)&*(blue)&*(blue)&*(blue)&*(blue)&*(blue)&*(blue)  \\
\none $\ba$&\none&*(foge)&*(foge)&*(foge)&*(foge)&*(foge)&*(foge)&*(foge)&*(foge)&*(foge)&*(foge)&*(foge)&*(foge)&*(foge)&*(foge)&*(foge)&*(foge)&*(foge) \\
\none $\ba$&\none&*(foge)&*(foge)&*(foge)&*(foge)&*(foge)&*(foge)&*(foge)&*(foge)&*(foge)&*(foge)&*(foge)\\
\none $\ab$&\none&*(blue)&*(blue)&*(blue)&*(blue)&*(blue)&*(blue)&*(blue)&*(blue)&*(blue)&*(blue)\\
\none $\ba$&\none&*(foge)&*(foge)&*(foge)&*(foge)&*(foge)&*(foge)&*(foge)&*(foge)&*(foge)\\
\none $\ba$&\none&*(foge)&*(foge)&*(foge)&*(foge)&*(foge)&*(foge)&*(foge)&*(foge)\\
\none $\ab$&\none&*(blue)&*(blue)&*(blue)&*(blue)&*(blue)&*(blue)\\
\none $\ab$&\none&*(blue)&*(blue)&*(blue)&*(blue)&*(blue)\\
\none $\ba$&\none&*(foge)&*(foge)&*(foge)&*(foge)\\
\none $\ab$&\none&*(blue)&*(blue)&*(blue)&*(blue)
\end{ytableau}
 
\end{center}
\begin{itemize}
\item[\Soo:] First, we identify the \textbf{consecutive} troublesome pairs of parts, i.e $(\lambda_i,\lambda_{i+1})$ such that $\lambda_i\not\gg_c\lambda_{i+1}$, by taking consecutively the greatest pairs in terms of size, in such a way that they are disjoint. In our example, we have 
\begin{equation}
\lambda = 24_{\ab} + 17_{\ba}+ 11_{\ba}+\underline{10_{\ab}+9_{\ba}}+8_{\ba}+\underline{6_{\ab}+5_{\ab}}+\underline{4_{\ba}+4_{\ab}}\,\,\cdot
\end{equation}
Then we simply replace them by the corresponding parts in $\mathcal{E}'$ using \eqref{half}. We denote the resulting partition by $\lambda' = \lambda'_1+\cdots+\lambda'_{s'}$ with parts with the exact order by just replacing the pairs (parts are no longer ordered here). Our example gives
\begin{equation}
\lambda' = 24_{\ab}+17_{\ba}+11_{\ba}+19_{\abab}+8_{\ba}+11_{\aba}+8_{\baba}\,,
\end{equation}
\begin{center}
\begin{ytableau}
\none $\ab$&\none&*(blue)&*(blue)&*(blue)&*(blue)&*(blue)&*(blue)&*(blue)&*(blue)&*(blue)&*(blue)&*(blue)&*(blue)&*(blue)&*(blue)&*(blue)&*(blue)&*(blue)&*(blue)&*(blue)&*(blue)&*(blue)&*(blue)&*(blue)&*(blue)  \\
\none $\ba$&\none&*(foge)&*(foge)&*(foge)&*(foge)&*(foge)&*(foge)&*(foge)&*(foge)&*(foge)&*(foge)&*(foge)&*(foge)&*(foge)&*(foge)&*(foge)&*(foge)&*(foge) \\
\none $\ba$&\none&*(foge)&*(foge)&*(foge)&*(foge)&*(foge)&*(foge)&*(foge)&*(foge)&*(foge)&*(foge)&*(foge)\\
\none $\abab$&\none&*(blue)&*(blue)&*(blue)&*(blue)&*(blue)&*(blue)&*(blue)&*(blue)&*(blue)&*(blue)&*(foge)&*(foge)&*(foge)&*(foge)&*(foge)&*(foge)&*(foge)&*(foge)&*(foge)\\
\none $\ba$&\none&*(foge)&*(foge)&*(foge)&*(foge)&*(foge)&*(foge)&*(foge)&*(foge)\\
\none $\aba$&\none&*(blue)&*(blue)&*(blue)&*(blue)&*(blue)&*(blue)&*(blue)&*(blue)&*(blue)&*(blue)&*(blue)\\
\none $\baba$&\none&*(foge)&*(foge)&*(foge)&*(foge)&*(blue)&*(blue)&*(blue)&*(blue)\\
\end{ytableau}
\end{center}

\item[\Stt:]As long as there exists $i\in\lbrace 1,\dots,s'-1\rbrace$ such that 
$\lambda'_i,\lambda'_{i+1} \in \mathcal{O}\times\mathcal{E}
$
and $\lambda'_i\not\gg_c\lambda'_{i+1}$, we just replace them by 
$\Lambda(\lambda'_i,\lambda'_{i+1}) \in \mathcal{E}\times\mathcal{O}\,$
in this order. By \Prp{pr1}, this means that we replace a pair which doesn't respect the order $\gg_c$ by a new one which does.
If we proceed in our example by choosing the smallest $i$ at each step, we have
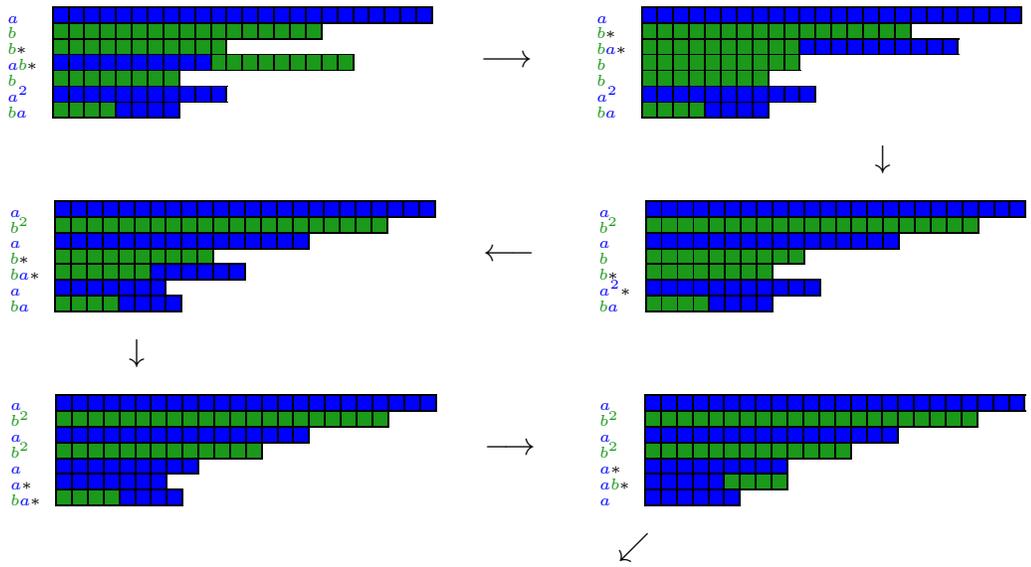
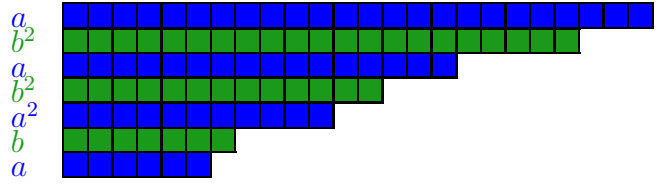
\ytableausetup{centertableaux,boxsize=0.2cm}
\begin{center}
\begin{tiny}
\begin{ytableau}
\none $\ab$&\none&*(blue)&*(blue)&*(blue)&*(blue)&*(blue)&*(blue)&*(blue)&*(blue)&*(blue)&*(blue)&*(blue)&*(blue)&*(blue)&*(blue)&*(blue)&*(blue)&*(blue)&*(blue)&*(blue)&*(blue)&*(blue)&*(blue)&*(blue)&*(blue)  \\
\none $\ba$&\none&*(foge)&*(foge)&*(foge)&*(foge)&*(foge)&*(foge)&*(foge)&*(foge)&*(foge)&*(foge)&*(foge)&*(foge)&*(foge)&*(foge)&*(foge)&*(foge)&*(foge) \\
\none $\ba\ast$&\none&*(foge)&*(foge)&*(foge)&*(foge)&*(foge)&*(foge)&*(foge)&*(foge)&*(foge)&*(foge)&*(foge)\\
\none $\abab\ast$&\none&*(blue)&*(blue)&*(blue)&*(blue)&*(blue)&*(blue)&*(blue)&*(blue)&*(blue)&*(blue)&*(foge)&*(foge)&*(foge)&*(foge)&*(foge)&*(foge)&*(foge)&*(foge)&*(foge)\\
\none $\ba$&\none&*(foge)&*(foge)&*(foge)&*(foge)&*(foge)&*(foge)&*(foge)&*(foge)\\
\none $\aba$&\none&*(blue)&*(blue)&*(blue)&*(blue)&*(blue)&*(blue)&*(blue)&*(blue)&*(blue)&*(blue)&*(blue)\\
\none $\baba$&\none&*(foge)&*(foge)&*(foge)&*(foge)&*(blue)&*(blue)&*(blue)&*(blue)\\
\end{ytableau}
\end{tiny}
$\quad\longrightarrow\quad$
\begin{tiny}
\begin{ytableau}
\none $\ab$&\none&*(blue)&*(blue)&*(blue)&*(blue)&*(blue)&*(blue)&*(blue)&*(blue)&*(blue)&*(blue)&*(blue)&*(blue)&*(blue)&*(blue)&*(blue)&*(blue)&*(blue)&*(blue)&*(blue)&*(blue)&*(blue)&*(blue)&*(blue)&*(blue)  \\
\none $\ba\ast$&\none&*(foge)&*(foge)&*(foge)&*(foge)&*(foge)&*(foge)&*(foge)&*(foge)&*(foge)&*(foge)&*(foge)&*(foge)&*(foge)&*(foge)&*(foge)&*(foge)&*(foge) \\
\none $\baba\ast$&\none&*(foge)&*(foge)&*(foge)&*(foge)&*(foge)&*(foge)&*(foge)&*(foge)&*(foge)&*(foge)&*(blue)&*(blue)&*(blue)&*(blue)&*(blue)&*(blue)&*(blue)&*(blue)&*(blue)&*(blue)\\
\none $\ba$&\none&*(foge)&*(foge)&*(foge)&*(foge)&*(foge)&*(foge)&*(foge)&*(foge)&*(foge)&*(foge)\\
\none $\ba$&\none&*(foge)&*(foge)&*(foge)&*(foge)&*(foge)&*(foge)&*(foge)&*(foge)\\
\none $\aba$&\none&*(blue)&*(blue)&*(blue)&*(blue)&*(blue)&*(blue)&*(blue)&*(blue)&*(blue)&*(blue)&*(blue)\\
\none $\baba$&\none&*(foge)&*(foge)&*(foge)&*(foge)&*(blue)&*(blue)&*(blue)&*(blue)\\
\end{ytableau}
\end{tiny}
\[\qquad\qquad\qquad\qquad\qquad\qquad\qquad\qquad\qquad\qquad\qquad\qquad\downarrow\]
\begin{tiny}
\begin{ytableau}
\none $\ab$&\none&*(blue)&*(blue)&*(blue)&*(blue)&*(blue)&*(blue)&*(blue)&*(blue)&*(blue)&*(blue)&*(blue)&*(blue)&*(blue)&*(blue)&*(blue)&*(blue)&*(blue)&*(blue)&*(blue)&*(blue)&*(blue)&*(blue)&*(blue)&*(blue)  \\
\none $\bab$&\none&*(foge)&*(foge)&*(foge)&*(foge)&*(foge)&*(foge)&*(foge)&*(foge)&*(foge)&*(foge)&*(foge)&*(foge)&*(foge)&*(foge)&*(foge)&*(foge)&*(foge)&*(foge)&*(foge)&*(foge)&*(foge)\\
\none $\ab$&\none&*(blue)&*(blue)&*(blue)&*(blue)&*(blue)&*(blue)&*(blue)&*(blue)&*(blue)&*(blue)&*(blue)&*(blue)&*(blue)&*(blue)&*(blue)&*(blue)\\
\none $\ba\ast$&\none&*(foge)&*(foge)&*(foge)&*(foge)&*(foge)&*(foge)&*(foge)&*(foge)&*(foge)&*(foge)\\
\none $\baba\ast$&\none&*(foge)&*(foge)&*(foge)&*(foge)&*(foge)&*(foge)&*(blue)&*(blue)&*(blue)&*(blue)&*(blue)&*(blue)\\
\none $\ab$&\none&*(blue)&*(blue)&*(blue)&*(blue)&*(blue)&*(blue)&*(blue)\\
\none $\baba$&\none&*(foge)&*(foge)&*(foge)&*(foge)&*(blue)&*(blue)&*(blue)&*(blue)\\
\end{ytableau}
\end{tiny}
$\quad\longleftarrow\quad$
\begin{tiny}
\begin{ytableau}
\none $\ab$&\none&*(blue)&*(blue)&*(blue)&*(blue)&*(blue)&*(blue)&*(blue)&*(blue)&*(blue)&*(blue)&*(blue)&*(blue)&*(blue)&*(blue)&*(blue)&*(blue)&*(blue)&*(blue)&*(blue)&*(blue)&*(blue)&*(blue)&*(blue)&*(blue)  \\
\none $\bab$&\none&*(foge)&*(foge)&*(foge)&*(foge)&*(foge)&*(foge)&*(foge)&*(foge)&*(foge)&*(foge)&*(foge)&*(foge)&*(foge)&*(foge)&*(foge)&*(foge)&*(foge)&*(foge)&*(foge)&*(foge)&*(foge)\\
\none $\ab$&\none&*(blue)&*(blue)&*(blue)&*(blue)&*(blue)&*(blue)&*(blue)&*(blue)&*(blue)&*(blue)&*(blue)&*(blue)&*(blue)&*(blue)&*(blue)&*(blue)\\
\none $\ba$&\none&*(foge)&*(foge)&*(foge)&*(foge)&*(foge)&*(foge)&*(foge)&*(foge)&*(foge)&*(foge)\\
\none $\ba\ast$&\none&*(foge)&*(foge)&*(foge)&*(foge)&*(foge)&*(foge)&*(foge)&*(foge)\\
\none $\aba\ast$&\none&*(blue)&*(blue)&*(blue)&*(blue)&*(blue)&*(blue)&*(blue)&*(blue)&*(blue)&*(blue)&*(blue)\\
\none $\baba$&\none&*(foge)&*(foge)&*(foge)&*(foge)&*(blue)&*(blue)&*(blue)&*(blue)\\
\end{ytableau}
\end{tiny}
\[\downarrow\qquad\qquad\qquad\qquad\qquad\qquad\qquad\qquad\qquad\qquad\qquad\qquad\]
\begin{tiny}
\begin{ytableau}
\none $\ab$&\none&*(blue)&*(blue)&*(blue)&*(blue)&*(blue)&*(blue)&*(blue)&*(blue)&*(blue)&*(blue)&*(blue)&*(blue)&*(blue)&*(blue)&*(blue)&*(blue)&*(blue)&*(blue)&*(blue)&*(blue)&*(blue)&*(blue)&*(blue)&*(blue)  \\
\none $\bab$&\none&*(foge)&*(foge)&*(foge)&*(foge)&*(foge)&*(foge)&*(foge)&*(foge)&*(foge)&*(foge)&*(foge)&*(foge)&*(foge)&*(foge)&*(foge)&*(foge)&*(foge)&*(foge)&*(foge)&*(foge)&*(foge)\\
\none $\ab$&\none&*(blue)&*(blue)&*(blue)&*(blue)&*(blue)&*(blue)&*(blue)&*(blue)&*(blue)&*(blue)&*(blue)&*(blue)&*(blue)&*(blue)&*(blue)&*(blue)\\
\none $\bab$&\none&*(foge)&*(foge)&*(foge)&*(foge)&*(foge)&*(foge)&*(foge)&*(foge)&*(foge)&*(foge)&*(foge)&*(foge)&*(foge)\\
\none $\ab$&\none&*(blue)&*(blue)&*(blue)&*(blue)&*(blue)&*(blue)&*(blue)&*(blue)&*(blue)\\
\none $\ab\ast$&\none&*(blue)&*(blue)&*(blue)&*(blue)&*(blue)&*(blue)&*(blue)\\
\none $\baba\ast$&\none&*(foge)&*(foge)&*(foge)&*(foge)&*(blue)&*(blue)&*(blue)&*(blue)\\
\end{ytableau}
\end{tiny}
$\quad\longrightarrow\quad$
\begin{tiny}
\begin{ytableau}
\none $\ab$&\none&*(blue)&*(blue)&*(blue)&*(blue)&*(blue)&*(blue)&*(blue)&*(blue)&*(blue)&*(blue)&*(blue)&*(blue)&*(blue)&*(blue)&*(blue)&*(blue)&*(blue)&*(blue)&*(blue)&*(blue)&*(blue)&*(blue)&*(blue)&*(blue)  \\
\none $\bab$&\none&*(foge)&*(foge)&*(foge)&*(foge)&*(foge)&*(foge)&*(foge)&*(foge)&*(foge)&*(foge)&*(foge)&*(foge)&*(foge)&*(foge)&*(foge)&*(foge)&*(foge)&*(foge)&*(foge)&*(foge)&*(foge)\\
\none $\ab$&\none&*(blue)&*(blue)&*(blue)&*(blue)&*(blue)&*(blue)&*(blue)&*(blue)&*(blue)&*(blue)&*(blue)&*(blue)&*(blue)&*(blue)&*(blue)&*(blue)\\
\none $\bab$&\none&*(foge)&*(foge)&*(foge)&*(foge)&*(foge)&*(foge)&*(foge)&*(foge)&*(foge)&*(foge)&*(foge)&*(foge)&*(foge)\\
\none $\ab\ast$&\none&*(blue)&*(blue)&*(blue)&*(blue)&*(blue)&*(blue)&*(blue)&*(blue)&*(blue)\\
\none $\abab\ast$&\none&*(blue)&*(blue)&*(blue)&*(blue)&*(blue)&*(foge)&*(foge)&*(foge)&*(foge)\\
\none $\ab$&\none&*(blue)&*(blue)&*(blue)&*(blue)&*(blue)&*(blue)\\
\end{ytableau}
\end{tiny}
\[\qquad\qquad\qquad\qquad\swarrow\]
\ytableausetup{centertableaux,boxsize=0.31cm}
 \begin{ytableau}
\none $\ab$&\none&*(blue)&*(blue)&*(blue)&*(blue)&*(blue)&*(blue)&*(blue)&*(blue)&*(blue)&*(blue)&*(blue)&*(blue)&*(blue)&*(blue)&*(blue)&*(blue)&*(blue)&*(blue)&*(blue)&*(blue)&*(blue)&*(blue)&*(blue)&*(blue)  \\
\none $\bab$&\none&*(foge)&*(foge)&*(foge)&*(foge)&*(foge)&*(foge)&*(foge)&*(foge)&*(foge)&*(foge)&*(foge)&*(foge)&*(foge)&*(foge)&*(foge)&*(foge)&*(foge)&*(foge)&*(foge)&*(foge)&*(foge)\\
\none $\ab$&\none&*(blue)&*(blue)&*(blue)&*(blue)&*(blue)&*(blue)&*(blue)&*(blue)&*(blue)&*(blue)&*(blue)&*(blue)&*(blue)&*(blue)&*(blue)&*(blue)\\
\none $\bab$&\none&*(foge)&*(foge)&*(foge)&*(foge)&*(foge)&*(foge)&*(foge)&*(foge)&*(foge)&*(foge)&*(foge)&*(foge)&*(foge)\\
\none $\aba$&\none&*(blue)&*(blue)&*(blue)&*(blue)&*(blue)&*(blue)&*(blue)&*(blue)&*(blue)&*(blue)&*(blue)\\
\none $\ba$&\none&*(foge)&*(foge)&*(foge)&*(foge)&*(foge)&*(foge)&*(foge)\\
\none $\ab$&\none&*(blue)&*(blue)&*(blue)&*(blue)&*(blue)&*(blue)\\
\end{ytableau}
\end{center}
We denote by $\lambda''$ the final result, which exists since the sum of the indices of parts in $\mathcal{E}$ strictly decreases by one at each step.
\end{itemize}
Then set $\Phi(\lambda)=\lambda''$. In our example, we obtain
\begin{equation}
\Phi(24_{\ab} + 17_{\ba}+ 11_{\ba}+10_{\ab}+9_{\ba}+8_{\ba}+6_{\ab}+5_{\ab}+4_{\ba}+4_{\ab})=24_{\ab}+21_{\bab}+16_{\ab}+13_{\bab}+11_{\aba}+7_{\ba}+6_{\ab}
\end{equation}
and we easily check that it belongs to $\mathcal{D}$. 
\\\\We will prove in \Sct{prf} that, during \Stt, a problem of order can only occur when $(\la'_i,\la'_{i+1})\in \Od\times \E$, and it means that, at each substep, 
\[\la'_i\not\gg_c\la'_{i+1}\Longrightarrow (\la'_i,\la'_{i+1})\in \Od\times \E\,\cdot\] 
Then, at the end of the process, the resulting partition will be well-ordered by $\gg_c$, and we will also show that its parts stay positive, so that it belongs to $\D$.
\subsubsection{How to compute $\Psi:\mathcal{D}\rightarrow\mathcal{C}$}\label{psii}
Let us take $\nu = \nu_1+\cdots+\nu_s \in \mathcal{D}$ with $\nu_1\gg_c\cdots\gg_c\nu_s$. We also take the example  $\nu = \lambda''$ in the previous part,
\begin{equation}
\nu=24_a+21_{b^2}+16_a+13_{b^2}+11_{a^2}+7_b+6_a\,\cdot
\end{equation}
\begin{itemize}
\item[\Soo:]As long as there exists $i\in\lbrace 1,\dots,s-1\rbrace$ such that
$(\nu_i,\nu_{i+1})\in \mathcal{E}\times\mathcal{O}$
and
$\mu(\nu_i)\not>_c \nu_{i+1}\,,$
we turn $(\nu_i,\nu_{i+1})$ into $\Lambda^{-1}(\nu_i,\nu_{i+1})\in \mathcal{O}\times\mathcal{E}$
. We denote the final result by $\nu'$, which exists since the sum of the indices of the parts in $\mathcal{O}$ strictly decreases at each step.
One can easily check that if we proceed by taking the greatest $i$ at each step, we have the exact reverse steps as we did before.
\ytableausetup{centertableaux,boxsize=0.2cm}
\begin{center}
\begin{tiny}
 \begin{ytableau}
\none $\ab$&\none&*(blue)&*(blue)&*(blue)&*(blue)&*(blue)&*(blue)&*(blue)&*(blue)&*(blue)&*(blue)&*(blue)&*(blue)&*(blue)&*(blue)&*(blue)&*(blue)&*(blue)&*(blue)&*(blue)&*(blue)&*(blue)&*(blue)&*(blue)&*(blue)  \\
\none $\bab$&\none&*(foge)&*(foge)&*(foge)&*(foge)&*(foge)&*(foge)&*(foge)&*(foge)&*(foge)&*(foge)&*(foge)&*(foge)&*(foge)&*(foge)&*(foge)&*(foge)&*(foge)&*(foge)&*(foge)&*(foge)&*(foge)\\
\none $\ab$&\none&*(blue)&*(blue)&*(blue)&*(blue)&*(blue)&*(blue)&*(blue)&*(blue)&*(blue)&*(blue)&*(blue)&*(blue)&*(blue)&*(blue)&*(blue)&*(blue)\\
\none $\bab$&\none&*(foge)&*(foge)&*(foge)&*(foge)&*(foge)&*(foge)&*(foge)&*(foge)&*(foge)&*(foge)&*(foge)&*(foge)&*(foge)\\
\none $\aba\ast$&\none&*(blue)&*(blue)&*(blue)&*(blue)&*(blue)&*(blue)&*(blue)&*(blue)&*(blue)&*(blue)&*(blue)\\
\none $\ba\ast$&\none&*(foge)&*(foge)&*(foge)&*(foge)&*(foge)&*(foge)&*(foge)\\
\none $\ab$&\none&*(blue)&*(blue)&*(blue)&*(blue)&*(blue)&*(blue)\\
\end{ytableau}
\end{tiny}
$\quad\longrightarrow\quad$
\begin{tiny}
\begin{ytableau}
\none $\ab$&\none&*(blue)&*(blue)&*(blue)&*(blue)&*(blue)&*(blue)&*(blue)&*(blue)&*(blue)&*(blue)&*(blue)&*(blue)&*(blue)&*(blue)&*(blue)&*(blue)&*(blue)&*(blue)&*(blue)&*(blue)&*(blue)&*(blue)&*(blue)&*(blue)  \\
\none $\bab$&\none&*(foge)&*(foge)&*(foge)&*(foge)&*(foge)&*(foge)&*(foge)&*(foge)&*(foge)&*(foge)&*(foge)&*(foge)&*(foge)&*(foge)&*(foge)&*(foge)&*(foge)&*(foge)&*(foge)&*(foge)&*(foge)\\
\none $\ab$&\none&*(blue)&*(blue)&*(blue)&*(blue)&*(blue)&*(blue)&*(blue)&*(blue)&*(blue)&*(blue)&*(blue)&*(blue)&*(blue)&*(blue)&*(blue)&*(blue)\\
\none $\bab$&\none&*(foge)&*(foge)&*(foge)&*(foge)&*(foge)&*(foge)&*(foge)&*(foge)&*(foge)&*(foge)&*(foge)&*(foge)&*(foge)\\
\none $\ab$&\none&*(blue)&*(blue)&*(blue)&*(blue)&*(blue)&*(blue)&*(blue)&*(blue)&*(blue)\\
\none $\abab\ast$&\none&*(blue)&*(blue)&*(blue)&*(blue)&*(blue)&*(foge)&*(foge)&*(foge)&*(foge)\\
\none $\ab\ast$&\none&*(blue)&*(blue)&*(blue)&*(blue)&*(blue)&*(blue)\\
\end{ytableau}
\end{tiny}
\[\qquad\qquad\qquad\qquad\qquad\qquad\qquad\qquad\qquad\qquad\qquad\qquad\downarrow\]
\begin{tiny}
\begin{ytableau}
\none $\ab$&\none&*(blue)&*(blue)&*(blue)&*(blue)&*(blue)&*(blue)&*(blue)&*(blue)&*(blue)&*(blue)&*(blue)&*(blue)&*(blue)&*(blue)&*(blue)&*(blue)&*(blue)&*(blue)&*(blue)&*(blue)&*(blue)&*(blue)&*(blue)&*(blue)  \\
\none $\bab$&\none&*(foge)&*(foge)&*(foge)&*(foge)&*(foge)&*(foge)&*(foge)&*(foge)&*(foge)&*(foge)&*(foge)&*(foge)&*(foge)&*(foge)&*(foge)&*(foge)&*(foge)&*(foge)&*(foge)&*(foge)&*(foge)\\
\none $\ab$&\none&*(blue)&*(blue)&*(blue)&*(blue)&*(blue)&*(blue)&*(blue)&*(blue)&*(blue)&*(blue)&*(blue)&*(blue)&*(blue)&*(blue)&*(blue)&*(blue)\\
\none $\ba$&\none&*(foge)&*(foge)&*(foge)&*(foge)&*(foge)&*(foge)&*(foge)&*(foge)&*(foge)&*(foge)\\
\none $\baba\ast$&\none&*(foge)&*(foge)&*(foge)&*(foge)&*(foge)&*(foge)&*(blue)&*(blue)&*(blue)&*(blue)&*(blue)&*(blue)\\
\none $\ab\ast$&\none&*(blue)&*(blue)&*(blue)&*(blue)&*(blue)&*(blue)&*(blue)\\
\none $\baba$&\none&*(foge)&*(foge)&*(foge)&*(foge)&*(blue)&*(blue)&*(blue)&*(blue)\\
\end{ytableau}
\end{tiny}
$\quad\longleftarrow\quad$
\begin{tiny}
\begin{ytableau}
\none $\ab$&\none&*(blue)&*(blue)&*(blue)&*(blue)&*(blue)&*(blue)&*(blue)&*(blue)&*(blue)&*(blue)&*(blue)&*(blue)&*(blue)&*(blue)&*(blue)&*(blue)&*(blue)&*(blue)&*(blue)&*(blue)&*(blue)&*(blue)&*(blue)&*(blue)  \\
\none $\bab$&\none&*(foge)&*(foge)&*(foge)&*(foge)&*(foge)&*(foge)&*(foge)&*(foge)&*(foge)&*(foge)&*(foge)&*(foge)&*(foge)&*(foge)&*(foge)&*(foge)&*(foge)&*(foge)&*(foge)&*(foge)&*(foge)\\
\none $\ab$&\none&*(blue)&*(blue)&*(blue)&*(blue)&*(blue)&*(blue)&*(blue)&*(blue)&*(blue)&*(blue)&*(blue)&*(blue)&*(blue)&*(blue)&*(blue)&*(blue)\\
\none $\bab\ast$&\none&*(foge)&*(foge)&*(foge)&*(foge)&*(foge)&*(foge)&*(foge)&*(foge)&*(foge)&*(foge)&*(foge)&*(foge)&*(foge)\\
\none $\ab\ast$&\none&*(blue)&*(blue)&*(blue)&*(blue)&*(blue)&*(blue)&*(blue)&*(blue)&*(blue)\\
\none $\ab$&\none&*(blue)&*(blue)&*(blue)&*(blue)&*(blue)&*(blue)&*(blue)\\
\none $\baba$&\none&*(foge)&*(foge)&*(foge)&*(foge)&*(blue)&*(blue)&*(blue)&*(blue)\\
\end{ytableau}
\end{tiny}
\[\downarrow\qquad\qquad\qquad\qquad\qquad\qquad\qquad\qquad\qquad\qquad\qquad\qquad\]
\begin{tiny}
\begin{ytableau}
\none $\ab$&\none&*(blue)&*(blue)&*(blue)&*(blue)&*(blue)&*(blue)&*(blue)&*(blue)&*(blue)&*(blue)&*(blue)&*(blue)&*(blue)&*(blue)&*(blue)&*(blue)&*(blue)&*(blue)&*(blue)&*(blue)&*(blue)&*(blue)&*(blue)&*(blue)  \\
\none $\bab\ast$&\none&*(foge)&*(foge)&*(foge)&*(foge)&*(foge)&*(foge)&*(foge)&*(foge)&*(foge)&*(foge)&*(foge)&*(foge)&*(foge)&*(foge)&*(foge)&*(foge)&*(foge)&*(foge)&*(foge)&*(foge)&*(foge)\\
\none $\ab\ast$&\none&*(blue)&*(blue)&*(blue)&*(blue)&*(blue)&*(blue)&*(blue)&*(blue)&*(blue)&*(blue)&*(blue)&*(blue)&*(blue)&*(blue)&*(blue)&*(blue)\\
\none $\ba$&\none&*(foge)&*(foge)&*(foge)&*(foge)&*(foge)&*(foge)&*(foge)&*(foge)&*(foge)&*(foge)\\
\none $\ba$&\none&*(foge)&*(foge)&*(foge)&*(foge)&*(foge)&*(foge)&*(foge)&*(foge)\\
\none $\aba$&\none&*(blue)&*(blue)&*(blue)&*(blue)&*(blue)&*(blue)&*(blue)&*(blue)&*(blue)&*(blue)&*(blue)\\
\none $\baba$&\none&*(foge)&*(foge)&*(foge)&*(foge)&*(blue)&*(blue)&*(blue)&*(blue)\\
\end{ytableau}
\end{tiny}
$\quad\longrightarrow\quad$
\begin{tiny}
\begin{ytableau}
\none $\ab$&\none&*(blue)&*(blue)&*(blue)&*(blue)&*(blue)&*(blue)&*(blue)&*(blue)&*(blue)&*(blue)&*(blue)&*(blue)&*(blue)&*(blue)&*(blue)&*(blue)&*(blue)&*(blue)&*(blue)&*(blue)&*(blue)&*(blue)&*(blue)&*(blue)  \\
\none $\ba$&\none&*(foge)&*(foge)&*(foge)&*(foge)&*(foge)&*(foge)&*(foge)&*(foge)&*(foge)&*(foge)&*(foge)&*(foge)&*(foge)&*(foge)&*(foge)&*(foge)&*(foge) \\
\none $\baba\ast$&\none&*(foge)&*(foge)&*(foge)&*(foge)&*(foge)&*(foge)&*(foge)&*(foge)&*(foge)&*(foge)&*(blue)&*(blue)&*(blue)&*(blue)&*(blue)&*(blue)&*(blue)&*(blue)&*(blue)&*(blue)\\
\none $\ba\ast$&\none&*(foge)&*(foge)&*(foge)&*(foge)&*(foge)&*(foge)&*(foge)&*(foge)&*(foge)&*(foge)\\
\none $\ba$&\none&*(foge)&*(foge)&*(foge)&*(foge)&*(foge)&*(foge)&*(foge)&*(foge)\\
\none $\aba$&\none&*(blue)&*(blue)&*(blue)&*(blue)&*(blue)&*(blue)&*(blue)&*(blue)&*(blue)&*(blue)&*(blue)\\
\none $\baba$&\none&*(foge)&*(foge)&*(foge)&*(foge)&*(blue)&*(blue)&*(blue)&*(blue)\\
\end{ytableau}
\end{tiny}
\[\qquad\qquad\qquad\qquad\swarrow\]
\ytableausetup{centertableaux,boxsize=0.31cm}
\begin{ytableau}
\none $\ab$&\none&*(blue)&*(blue)&*(blue)&*(blue)&*(blue)&*(blue)&*(blue)&*(blue)&*(blue)&*(blue)&*(blue)&*(blue)&*(blue)&*(blue)&*(blue)&*(blue)&*(blue)&*(blue)&*(blue)&*(blue)&*(blue)&*(blue)&*(blue)&*(blue)  \\
\none $\ba$&\none&*(foge)&*(foge)&*(foge)&*(foge)&*(foge)&*(foge)&*(foge)&*(foge)&*(foge)&*(foge)&*(foge)&*(foge)&*(foge)&*(foge)&*(foge)&*(foge)&*(foge) \\
\none $\ba$&\none&*(foge)&*(foge)&*(foge)&*(foge)&*(foge)&*(foge)&*(foge)&*(foge)&*(foge)&*(foge)&*(foge)\\
\none $\abab$&\none&*(blue)&*(blue)&*(blue)&*(blue)&*(blue)&*(blue)&*(blue)&*(blue)&*(blue)&*(blue)&*(foge)&*(foge)&*(foge)&*(foge)&*(foge)&*(foge)&*(foge)&*(foge)&*(foge)\\
\none $\ba$&\none&*(foge)&*(foge)&*(foge)&*(foge)&*(foge)&*(foge)&*(foge)&*(foge)\\
\none $\aba$&\none&*(blue)&*(blue)&*(blue)&*(blue)&*(blue)&*(blue)&*(blue)&*(blue)&*(blue)&*(blue)&*(blue)\\
\none $\baba$&\none&*(foge)&*(foge)&*(foge)&*(foge)&*(blue)&*(blue)&*(blue)&*(blue)\\
\end{ytableau}
\end{center}
And then 
\begin{equation}
\nu' = 24_{\ab}+17_{\ba}+11_{\ba}+19_{\abab}+8_{\ba}+11_{\aba}+8_{\baba} =\lambda'\,\cdot
\end{equation}
\item[\Stt:]We finish by dividing all parts in $\mathcal{E}$ into their upper and lower halves and keeping the order. We finally obtain  $\nu''$ and
we set $\Psi(\nu)=\nu''$.
\end{itemize}
In our example, we obtain 
\begin{equation}
\nu''= 24_{\ab} + 17_{\ba}+ 11_{\ba}+10_{\ab}+9_{\ba}+8_{\ba}+6_{\ab}+5_{\ab}+4_{\ba}+4_{\ab} =\lambda\,\cdot
\end{equation}
and then
\begin{equation}
\Psi(24_{\ab}+21_{\bab}+16_{\ab}+13_{\bab}+11_{\aba}+7_{\ba}+6_{\ab})= 24_{\ab} + 17_{\ba}+ 11_{\ba}+10_{\ab}+9_{\ba}+8_{\ba}+6_{\ab}+5_{\ab}+4_{\ba}+4_{\ab}\,\cdot
\end{equation}
We will discuss the uniqueness of the final result, its belonging to $\mathcal{C}$, and the fact that $\Psi=\Phi^{-1}$ in \Sct{prf}.

\section{Proof of the well-definedness of bijections $\Phi$ and $\Psi$}\label{prf}
In the next two subsections, we will show that $\Phi$ and $\Psi$ are well-defined.
\subsection{Well-definedness of $\Phi$}\label{prfphi}
\begin{prop}\label{prphi}
 For any $\la\in \C$, the final result after \St is unique and belongs to $\D$. Moreover, the result is independent of the order in which we proceed in \St (choices of unordered parts)
\end{prop}
Let us take any $\la = \la_1+\cdots+\la_s \in \C$, and set $\la_i = (l_i,c_i)\in \Od'$ for all $i\in\sss$. We also define a function $\Delta$ on $\sss^2$ as follows,
\begin{equation}\label{del}
 \Delta : \,\,(i,j) \mapsto \left\lbrace \begin{array}{l c l}
                                     0& \text{if}& i=j\\
                                     \displaystyle\sum_{k=i}^{j-1} \chi(c_k\leq c_{k+1})& \text{if}& i<j\\
                                     \displaystyle-\sum_{k=j}^{i-1} \chi(c_k\leq c_{k+1})& \text{if}& i>j\\
                                    \end{array}
                       \right. \,,
\end{equation}
so that $\Delta$ satisfies \Crr : $\Delta(i,k)+\Delta(k,j) = \Delta(i,j)$. We can also remark that, for any $i\leq j$,
\begin{equation}\label{signe}
 0\leq \chi(c_i\leq c_j)\leq \Delta(i,j)\leq j-i\,\,, \,\Delta(j,i)=-\Delta(i,j)\,\cdot
\end{equation}
Since $\la$ is well-ordered by $>_c$, we then have for all $i\in \ssss$, 
\begin{equation}\label{rang}
\la_i>_c\la_{i+1} \Longleftrightarrow l_i-l_{i+1}\geq \Delta(i,i+1)\,\cdot
\end{equation}
At \Soo, the choice of greatest troublesome pairs is formalized as follows : 
\begin{itemize}
 \item $i_1$ is the smallest $i\in \ssss$ such that $l_i-l_{i+1} = \Delta(i,i+1)$,
 \item if $i_{k-1}$ is chosen, then, until it is possible, $i_k$ is the smallest $i\in \{i_{k-1}+2,\ldots,s-1\}$ such that $l_i-l_{i+1} = \Delta(i,i+1)$,
\end{itemize}
Thus, $I= \{i_k\}$ can be viewed as the set of indices of upper halves, $I+1$ the set of indices for lower halves and $J = \sss\setminus(I\sqcup (I+1))$ the set of indices of parts that stay in $\Od'$.  
In fact, $I,J$ are the unique sets satisfying the following conditions :
\begin{enumerate}
\item $I,I+1$ and $J$ form a set-partition of $\sss$,
\item for all $i\in I$, $\la_i$ and $\la_{i+1}$ are consecutive for $>_c$, which means that $l_i-l_{i+1}=\Delta(i,i+1)$,
\item for all $j\in J\cap \ssss$, $\la_j\gg_c \la_{j+1}$, or equivalently,  
\begin{equation}\label{+}
l_j-l_{j+1}\geq 1+\Delta(j,j+1)\,\cdot
\end{equation}
\end{enumerate}
If we define $\alpha$ on $\sss^2$ by 
\begin{equation}
 \alpha : (i,j) \mapsto \left\lbrace \begin{array}{l c l}
                                     | [i,j)\cap J|& \text{if}& i\leq j\\
                                     -\alpha(j,i)& \text{otherwise}&\\
                                    \end{array}
                       \right. \,,
\end{equation}
then $\alpha$ satisfies \Crr, and, by \eqref{rang} and \eqref{+}, we have for all $i\leq j \in \sss$ 
\begin{equation}\label{ecart}
 l_i-l_j \geq \alpha(i,j)+\Delta(i,j)\,\cdot
\end{equation}
Then, at the end of \Soo, parts in $\E'$  are $\la_{i}+\la_{i+1}$ for $i\in I$, and parts in $\Od'$ are $\la_j$ for $j\in J$. 
\\With example \eqref{exp}, $\lambda = 24_{\ab} + 17_{\ba}+ 11_{\ba}+10_{\ab}+9_{\ba}+8_{\ba}+6_{\ab}+5_{\ab}+4_{\ba}+4_{\ab}$, we have $s=10$ and 
\[
\begin{array}{|c|c|c|c|c|c|c|c|c|c|c|}
 \hline
 i&1&2&3&4&5&6&7&8&9&10\\
 \hline
 c_i&\ab&\ba&\ba&\ab&\ba&\ba&\ab&\ab&\ba&\ab\\
 \hline
 l_i&24&17&11&10&9&8&6&5&4&4\\
 \hline
 set&J&J&J&I&I+1&J&I&I+1&I&I+1\\
 \hline
 \Delta(1,i)&0&1&2&2&3&4&4&5&6&6\\
 \hline
 \alpha(1,i)&0&1&2&3&3&3&4&4&4&4\\
 \hline
\end{array}
\]
\\\\The key question is then how do positions of parts in $\E'$ and $\Od'$ evolve during \Stt. We can define a position bijection $P$ on $\sss$, which indicates new indices of original parts after some applications of $\Lambda$ (parts in $\E'$ have two indices), and  we have:
\begin{itemize}
 \item $P(i+1)=P(i)+1$ for all $i\in I$, since the upper and lower halves move together,
 \item $P$ is increasing on $I\sqcup(I+1)$ and on $J$ since parts of the same kind never cross,
 \item $P(i)\leq i$ for all $i\in I$ and $P(j)\geq j$ for all $j\in J$.
 \end{itemize}
 \begin{rem}\label{rem1}
Since $P$ is a permutation of $\sss$ and is increasing on $I\sqcup (I+1)$ and $J$, then 
$P(I)$ determines $P(I+1)$ and $P(J)$.
\end{rem}
\begin{prop}\label{pr11}
 Let $\phi$ be the function on $J\times I$ defined by 
 \begin{equation}
  \phi : (j,i) \,\mapsto \, l_j-2l_{i+1}-\Delta(j,i+1)-\Delta(i+1-\alpha(j,i),i+1)\,\cdot
 \end{equation}
Then for any $i\in I$, the final position (after \Stt) of the original part $\la_i+\la_{i+1}$ is 
\begin{equation}
 P(i) = i - |\{j\in J/ j<i\,, \phi(j,i)<0\}|\,\cdot
\end{equation}
\end{prop}
\begin{prop}\label{pr12}
 The final result after \St is in $\D$.
\end{prop}
In our example, we have the following table for $\phi$:
\[
\begin{array}{|c|c|c|c|}
 \hline
 _{j\in J}\setminus^{i\in I}&4&7&9\\
 \hline
 1&1&6&8\\
 \hline
 2&-4&1&2\\
 \hline
 3&-9&-3&-2\\
 \hline
 6&-9&-4&-2\\
 \hline
\end{array}
\]
and then the final position 
\[\begin{array}{|c|c|c|c|}
   \hline
   i&4&7&9\\
   \hline
   P(i)&2&5&7\\
   \hline
  \end{array}\quad,\quad
  \begin{array}{|c|c|c|c|}
   \hline
   i+1&5&8&10\\
   \hline
   P(i+1)&3&6&8\\
   \hline
  \end{array}\quad,\quad 
  \begin{array}{|c|c|c|c|c|}
   \hline
   j&1&2&3&6\\
   \hline
   P(j)&1&4&9&10\\
   \hline
  \end{array}\quad,
\]
and it matches with the final result 
$\Phi(\la) = 24_{\ab}+21_{\bab}+16_{\ab}+13_{\bab}+11_{\aba}+7_{\ba}+6_{\ab}$.
\\Before proving \Prp{pr11} and \Prp{pr12},
we state and prove two crucial lemmas.
\begin{lem}\label{lem1}
 If the original part $\la_k=(l_k,c_k)$ at position $k$ moves to position $P(k)$, then it becomes $\la'_{P(k)}=(\Delta(P(k),k)+l_k,c_{P(k)})$.
\end{lem}
\begin{proof}
We prove this recursively by using \Cr and observing what happens under the transformation $\Lambda$, at position $k$ (with $\la'_k,\la'_{k+1}+\la'_{k+2}\in \Od\times \E$),  and using \eqref{mov}:
\[\begin{array}{|c | c |c |c|}
\hline
   \text{positions}& k&k+1&k+2 \\
   \hline
   \text{colors}& c_k&c_{k+1}&c_{k+2} \\
   \hline
   \text{part sizes before }\Lambda&l'_k&l'_{k+1}&l'_{k+2}\\
   \hline
   \text{part sizes after }\Lambda&\Delta(k,k+1)+l'_{k+1}&\Delta(k+1,k+2)+l'_{k+2}&\Delta(k+2,k)+l'_k\\
   \hline
  \end{array}\,\cdot
\]
\end{proof}
\begin{lem}\label{lem2} 
The function $\phi$ is decreasing according to $J$ and increasing according to $I$.
 \end{lem}
\begin{proof}
$\quad$
\begin{itemize}
 \item For any $j<j'\in J$ and $i\in I$, we have by \Cr
 \begin{align*}
 \phi(j,i)-\phi(j',i) &= l_j-l_{j'} - \Delta(j,j')-\Delta(i+1-\alpha(j,i),i+1-\alpha(j',i))\,\,\quad\\
 &\geq \alpha(j,j')-\Delta(i+1-\alpha(j,i),i+1-\alpha(j',i)) \quad \text{(by \eqref{ecart}).}
 \end{align*}
 But \Cr gives
 \[i+1-\alpha(j',i)-(i+1-\alpha(j,i)) = \alpha(j,j')\geq 0\,,\]
 so that, by \eqref{signe}, we obtain $\phi(j,i)-\phi(j',i)\geq 0$ .
 \item For any $j\in J$ and $i<i'\in I$, we have by \Cr
 \begin{align*}
 \phi(j,i')-\phi(j,i) &=2(l_{i+1}-l_{i'+1})-\Delta(i+1,i'+1)+\Delta(i+1-\alpha(j,i),i+1)\\
 &\quad +\Delta(i'+1,i'+1-\alpha(j,i'))\\
 &=2(l_{i+1}-l_{i'+1}-\Delta(i+1,i'+1))+\Delta(i+1-\alpha(j,i),i'+1-\alpha(j,i'))\\
 &\geq 2\alpha(i+1,i'+1)+ \Delta(i+1-\alpha(j,i),i'+1-\alpha(j,i'))\,\cdot
 \end{align*}
 Since \Cr gives
 \[i'+1-\alpha(j,i')-(i+1-\alpha(j,i)) = i'-i-\alpha(i,i') = |[i,i')\cap (I\sqcup I+1)|\geq 0\,,\]
 by \eqref{signe}, we then have $\phi(j,i')-\phi(j,i)\geq 0$. 
\end{itemize}
\end{proof}
\begin{proof}[Proof of \Prp{pr11}]
Now, let $P$ be the final position. 
\begin{itemize}
 \item Let us suppose that there exist $j,i\in J\times I$ such that $j<i, P(j)<P(i)$ and $\phi(j,i)<0$. By \Lem{lem2} we have that $\phi(j',i')<0$ for all $j\leq j'\in J,\,\,i\geq i' \in I$.
 Also since $P$ is a bijection on $\sss$ and increasing on $J$ and $I$, and $P(J)+1\setminus P(J) \subset P(I)$, we necessarily have some $j\leq j'\in J,\,\,i\geq i' \in I$ such that $P(j')+1 = P(i')$.
 We can also observe that $j'<i'$, since $P(j')\geq j'$ and $P(i')\leq i'$.
 But we obtain by \Lem{lem1} the following difference in part sizes:
 \begin{align*}
 D &= \la'_{P(j')}-(\la'_{P(j')+1}+\la'_{P(j')+2}) - \Delta(P(j'),P(j')+2)\\
 &=l_{j'}+\Delta(P(j'),j') - [ 2 (l_{i'+1}+\Delta(P(j')+2,i'+1))+\Delta(P(j')+1,P(j')+2)]\\
 &\quad-\Delta(P(j'),P(j')+2)\\
 &= l_{j'}-2l_{i'+1} -(\Delta(j',P(j'))+\Delta(P(j'),P(j')+2)+\Delta(P(j')+2,i'+1))\\
 &\quad -(\Delta(P(j')+1,P(j')+2)+\Delta(P(j')+2,i'+1))\\
 &= l_{j'}-2l_{i'+1} - \Delta(j',i'+1) - \Delta(P(j')+1,i'+1)\,\cdot
 \end{align*}
Now, what exactly is $P(j')$? Since $P$ is increasing on $J$ and $I\sqcup (I+1)$, and $P(j')+1 = P(i')$, we exactly have
\[P(j') = |[1,j']\cap J|+|[1,i')\cap (I\sqcup I+1)| = 1+\alpha(1,j')+i'-1 -\alpha(1,i') = i'-\alpha(j',i')\,\cdot\]
Finally, we obtain that
\[D = l_{j'}-2l_{i'+1} - \Delta(j',i'+1) - \Delta(i'+1-\alpha(j',i'),i'+1) = \phi(j',i')<0\,\cdot\]
The difference $D$ is negative, and by \eqref{oe}, this implies that $\la'_{P(j')}\not \gg_c \la'_{P(j')+1}+\la'_{P(j')+2}$, so that $P$ is no longer the final position.
\\The final position is then such that $P(i)<P(j)$ for all $(j,i)\in J\times I$ with $j<i$ and $\phi(j,i)<0$.
\item As soon as we cross all pairs $(j,i)\in J\times I$ with $j<i$ and $\phi(j,i)<0$, we can no longer cross. In fact, if we have $P(j)+1 = P(i)$, then necessarily $j<i$ and $\phi(j,i)\geq 0$,
and the previous paragraph told us that the sizes' difference (minus $\Delta(P(j),P(j)+2)$) is exactly $\phi(j,i)$,  so that by \eqref{oe}, $\la'_{P(j)}\gg_c \la'_{P(j)+1}+\la'_{P(j)+2}$.
\end{itemize}
In conclusion, the final position for each $i\in I$ is such that 
\[i-P(i) = |\{j\in J/ j<i\,,\phi(j,i)<0\}|\,\cdot\]
\end{proof}
\begin{proof}[Proof of \Prp{pr12}]
Recall that, by \Lem{lem1}, the (primary) part originally at position $k$  becomes 
$\la'_{P(k)}=(\Delta(P(k),k)+l_k,c_{P(k)})$.
\begin{enumerate}
 \item Parts remain in $\Od'$ and $\E'$.
\begin{enumerate}
 \item For any $i\in I$, we have that $P(i)+1\leq i+1$, so that by \eqref{signe}, \[l_{i+1}+\Delta(P(i)+1,i+1)\geq l_{i+1}>0\,\cdot\]
 \item For any $j\in J$, we have that $P(j)\geq j$, so that by \eqref{signe} and \eqref{ecart} \[l_{j}+\Delta(P(j),j)= l_j - \Delta(j,P(j))\geq \alpha(j,P(j))+ l_{P(j)}>0\,\cdot\]
\end{enumerate}
\item Parts of the same kind are well ordered by $\gg_c$. 
\begin{enumerate}
 \item First, we consider two parts in $\Od'$. For any $j<j'\in J$, we have that $P(j)<P(j')$, and by \Crr, 
 \begin{align*}
 l_j + \Delta(P(j),j) - (l_{j'}+\Delta(P(j'),j')) &= (l_j-l_{j'}-\Delta(j,j'))+\Delta(P(j),P(j'))\\
 &\geq \alpha(j,j')+\Delta(P(j),P(j'))\quad\text{(by \eqref{ecart})}\\
 &\geq 1 + \chi(c_{P(j)}\leq c_{P(j')})\quad\text{(by \eqref{signe})}\,\cdot
 \end{align*}
 By \eqref{oo}, we conclude that $\la'_{P(j)}\gg_c \la'_{P(j')}$.
 \item Now, we consider two parts in $\E'$. For any $i<i'\in I$, we have that $P(i+1)= P(i)+1<P(i')$, and arguing as above, we have
 \begin{align*}
 l_{i+1} + \Delta(P(i+1),i+1) - (l_{i'}+\Delta(P(i'),i')) &\geq \alpha(i+1,i')+\Delta(P(i+1),P(i'))\\
 &\geq \chi(c_{P(i+1)}\leq c_{P(i')})\quad \text{(by \eqref{signe})}.
\end{align*}  
 By using \eqref{gm} and \eqref{odd}, we obtain that $\la'_{P(i)}+\la'_{P(i)+1}\gg_c \la'_{P(i')}+\la'_{P(i')+1}$.
\end{enumerate}
\item Finally, we show that parts of different kind are well-ordered.
By \Prp{pr1} and \Prp{pr11}, we can see that, for any $(j,i)\in J\times I$ such that $j<i$, parts
$\la'_{P(j)}$ and $\la'_{P(i)}+\la'_{P(i)+1}$ are well-ordered by $\gg_c$. 
\\Let us now consider the case $i+1<j$. We necessarily have that
$P(i)+1\leq i+1<j\leq P(j)$ so that $P(i)+1<P(j)$. We then obtain 
\begin{align*}
D&=2(l_{i+1}+\Delta(P(i)+1,i+1))+\Delta(P(i),P(i)+1)-(l_j-\Delta(j,P(j)))\\
&= l_{i+1}+\Delta(P(i)+1,i+1) + \Delta(P(i),i+1)+(l_{i+1}-l_j)+\Delta(j,P(j))\\
&\geq l_{i+1}+\Delta(P(i)+1,i+1) + \Delta(P(i),P(j))+\alpha(i+1,j)\quad\text{(by  \eqref{ecart})}\\
&\geq 1 + \Delta(P(i),P(j))\geq 1+\chi(c_{P(i)}\leq c_{P(i)+1})+\chi(c_{P(i)+1}\leq c_{P(j)})\quad\text{(by  \eqref{signe})}.
\end{align*}
This means by \eqref{eo} that $\la'_{P(i)}+\la'_{P(i)+1}\gg_c\la'_{P(j)}$ for all $j>i$.
\end{enumerate}
To conclude, we always have that, for any $k,k'\in J\sqcup I$, if $P(k)<P(k')$, then the corresponding parts in $\Od',\E'$ are well-ordered by $\gg_c$.
\end{proof}
\Prp{prphi} follows immediately from \Prp{pr11} and \Prp{pr12}.
\subsection{Well-definedness of $\Psi$}\label{prfpsi}
\begin{prop}\label{prpsi}
 For any $\nu\in \D$, the final result after \St is unique and belongs to $\C$.
 Moreover, the result is independent of the order in which we proceed in \Soo.
\end{prop}
Let us consider any $\nu = \nu_1+\cdots+\nu_t \in \D$.
We now set $\nu = \nu'_1+\cdots+\nu'_s$, where we represent all primary parts that appear in $\nu$, by counting both upper and lower halves of each part in $\E'$.  We will then have as before sets 
$J,I$ respectively for indices of parts in $\Od',\E'$  such that $J\sqcup I \sqcup (I+1) = \sss$, where $s = t + \text{number of parts in }\E'$.  We then have $\nu'_k=(l_k,c_k)$, and the parts originally in $\Od'$ are 
$\nu'_j=(l_j,c_j)$ for all $j\in J$, and those originally in $\E'$ are 
$\nu'_i+\nu'_{i+1}=(l_{i+1},c_i,c_{i+1})$ for all $i\in I$.
\\\\We define $\Delta$ on $\sss^2$ in the same way we did before in \eqref{del}. Since during \textbf{Step 1} of $\Psi$, we apply $\Lambda^{-1}$, then for any position permutation $Q$ of $\sss$,
we have that $Q(i+1)=Q(i)+1$  and $Q(i)\geq i$ for all $i\in I$,
$Q$ is increasing on $I$ and $J$, $Q(j)\leq j$ for all $j\in J$ (see Remark \ref{rem1}), and finally, an analogue of \Lem{lem1} follows from \eqref{mov1}.
\begin{lem}\label{lem21}
 If the original part $\nu'_k=(l_k,c_k)$ at position $k$ moves to position $Q(k)$, then it becomes $\nu''_{Q(k)}=(\Delta(Q(k),k)+l_k,c_{Q(k)})$.
\end{lem}
\begin{prop}\label{pr21}
 Let $\psi$ be the function on $J\times I$ defined by 
 \begin{equation}
  \psi : (j,i) \,\mapsto \, l_j-l_i-\Delta(j,i)\,\cdot
 \end{equation}
Then for any $i\in I$, the final position (after \Soo) of the part originally at position $i$ is 
\begin{equation}
 Q(i) = i + |\{j\in J/ j>i\,, \psi(j,i)>0\}|\,\cdot
\end{equation}
\end{prop}
\begin{prop}\label{pr22}
 The final result after \textbf{Step 2} is in $\C$.
\end{prop}
Before proving these, let us first consider the function $\beta$ on $\sss^2$ as follows:
\begin{equation}\label{difff}
 \beta : \,\,(i,j) \mapsto \left\lbrace \begin{array}{l c l}                                   
                                     | (i,j]\cap J|& \text{if}& i\leq j\\
                                     -\beta(j,i)& \text{otherwise}&\\
                                    \end{array}
                       \right. \,,
\end{equation}
and we can easily see that $\beta$ satisfies \Crr.
We now state an important lemma.
\begin{lem}\label{lem22} Let us set 
\[\l'_k = \left\lbrace \begin{array}{l c l}
                        l_k &\text{if}& k\in J\\
                        2l_k&\text{if}&k\in I\sqcup (I+1)
                       \end{array}
              \right.\,\cdot
 \]
 Then for all $k\leq k'\in \sss$, we have 
 \begin{equation}\label{ecart2}
 l'_k-l'_{k'} \geq \beta(k,k')+\Delta(k,k')\,\cdot
 \end{equation}
Morever, for all $i\leq i' \in I\sqcup (I+1)$, we have 
\begin{equation}\label{ecart3}
 l_i-l_{i'}\geq \Delta(i,i')\,\cdot
\end{equation}
\end{lem}
\begin{proof}
Since the functions $\beta$ and $\Delta$ satisfy \Crr, in order to show \eqref{ecart2}, we just need to prove that for all $k\in \ssss$,
\[l'_k-l'_{k+1}\geq \beta(k,k+1)+\Delta(k,k+1)\,\cdot\]
\begin{itemize}
 \item If $k\in I$, then $k+1\in I+1$ and 
 \[l'_k-l'_{k+1} = 2\Delta(k,k+1)\geq \Delta(k,k+1) = \beta(k,k+1)+\Delta(k,k+1)\,\cdot\]
 \item If $k\in I+1$ and $k+1\in I$, then by \eqref{ee},
 \[(l_k,c_{k-1},c_k)\gg_c(l_{k+2},c_{k+1},c_{k+2}) \Leftrightarrow l'_k-l'_{k+1} \geq  2\Delta(k,k+1)\geq \beta(k,k+1)+\Delta(k,k+1)\,\cdot\]
 \item If $k\in I+1$ and $k+1\in J$, then by \eqref{eo},
 \[(l_k,c_{k-1},c_k)\gg_c(l_{k+1},c_{k+1}) \Leftrightarrow l'_k-l'_{k+1} \geq  1+ \Delta(k,k+1)=\beta(k,k+1)+\Delta(k,k+1)\,\cdot\]
 \item If $k\in J$ and $k+1\in I$, then by \eqref{oe},
 \[(l_k,c_k)\gg_c(l_{k+2},c_{k+1},c_{k+2}) \Leftrightarrow l'_k-l'_{k+1} \geq  \Delta(k,k+1)=\beta(k,k+1)+\Delta(k,k+1)\,\cdot\]
 \item If $k,k+1\in J$, then by \eqref{oo},
 \[(l_k,c_k)\gg_c(l_{k+1},c_{k+1}) \Leftrightarrow l'_k-l'_{k+1} \geq 1+\Delta(k,k+1)=\beta(k,k+1)+\Delta(k,k+1)\,\cdot\]
\end{itemize}
To show for \eqref{ecart3}, we only need to prove it for two consecutive $i,i' \in I\sqcup I+1$. It is easy for $i\in I$, since the following index is $i+1\in I+1$, and 
$l_i-l_{i+1}=\Delta(i,i+1)$. Now let us take $i\in I+1$. The next $i'$ (if it exists) must necessarily be in $I$, and by \eqref{ecart2}, we have by definition of $\beta$ and 
\eqref{signe} that
\begin{align*}
2(l_i-l_{i'})&= l'_i-l'_{i'}\\
& \geq \beta(i,i')+\Delta(i,i') \\
&= i'-i-1+\Delta(i,i')\\
&\geq 2\Delta(i,i')-1\,
\end{align*}
and this implies that
$ l_i-l_{i'}\geq \Delta(i,i')- \frac{1}{2}$, 
and since $l_i-l_{i'}$ is an integer, we then have
$l_i-l_{i'}\geq \Delta(i,i')$.
\end{proof}
\begin{proof}[Proof of \Prp{pr21}]
 With \Crr, we can easily see that $\psi$ is decreasing according to $J$ (by using \eqref{ecart2}), and increasing according to $I$ (by \eqref{ecart3}). Let  $Q$ be the final position of \Soo.
 \begin{itemize}
  \item Suppose that there exist $(j,i)\in J\times I$ such that $j>i,\, \psi(j,i)>0$ but $Q(j)>Q(i)$.
  Then by the same reasoning as in the proof of \Prp{pr11}, there exist $i\leq i'<j'\leq j$ such that $Q(i')+2=Q(j')$ (since $Q(J)-1\setminus Q(j)\subset Q(I)+1$). We also have $\psi(j',i')>0$.
  But 
  \begin{align*}
  0&<\psi(j',i')\\
  &= l_{j'}-l_{i'}-\Delta(j',i')\\
  &= l_{j'}-l_{i'+1}-\Delta(i',i'+1)-\Delta(j',Q(j'))-\Delta(Q(j'),Q(i')+1)-\Delta(Q(i')+1,i')\\
  &= (l_{j'}+\Delta(Q(j'),j'))-(l_{i'+1}+\Delta(Q(i')+1,i'+1))-\Delta(Q(j'),Q(i')+1)\,,  
  \end{align*}
  so that, by \eqref{odd}, $\nu''_{Q(i')+1} = \mu(\nu''_{Q(i')}+\nu''_{Q(i')+1})\not >_c \nu''_{Q(j')}$, and then we can still apply $\Lambda^{-1}$, and $Q$ is no longer the final position.
  \\\\We must then cross all pairs $(j,i)\in J\times I$ such that $j>i,\, \psi(j,i)>0$ before reaching the final position.
  \item After crossing all such pairs, we cannot cross anymore.  In fact, if  $(j,i)\in J\times I$ such  that $Q(i)+2 = Q(j)$, then $j>i,\, \psi(j,i)\leq 0$, and a calculation as above shows that $\nu''_{Q(i)+1} = \mu(\nu''_{Q(i)}+\nu''_{Q(i)+1}) >_c \nu''_{Q(j)}$.
 \end{itemize}
\end{proof}
\begin{proof}[Proof of \Prp{pr22}]
$\quad$
\begin{enumerate}
 \item Parts remain in $\Od'$.
\begin{enumerate}
 \item For any $j\in J$, we have that $Q(j)\leq j$, so that by \eqref{signe}, \[l_{j}+\Delta(Q(j),j)\geq l_{j}>0\,\cdot\]
 \item Since $Q$ is increasing on $I\sqcup (I+1)$, we just need to check for the last $i+1\in I+1$ that $l_{i+1}-\Delta(i+1,Q(i+1))>0$.
 But we have by \eqref{ecart2} and \eqref{signe} that
 \begin{align*}
 2l_{i+1}-l'_s &\geq \beta(i+1,s)+\Delta(i+1,s)\\
 &= s-(i+1)+\Delta(i+1,s)\\
 &\geq 2\Delta(i+1,s)\,,
 \end{align*}
 so that
 \[l_{i+1}-\Delta(i+1,Q(i+1))\geq \Delta(Q(i+1),s)+\frac{1}{2}l'_s>0\,\cdot\]
\end{enumerate}
\item Parts coming the same kind are well-ordered. 
\begin{enumerate}
 \item For any $j<j'\in J$, we have that $Q(j)<Q(j')$, and by \Cr and \eqref{ecart2}, 
 \begin{align*}
 l_j + \Delta(Q(j),j) - (l_{j'}+\Delta(Q(j'),j')) &\geq \beta(j,j')+\Delta(Q(j),Q(j'))\\
 &\geq 1 + \chi(c_{Q(j)}\leq c_{Q(j')})\,\cdot
 \end{align*}
 Then, by \eqref{oo}, we obtain  $\nu''_{Q(j)}\gg_c\nu''_{Q(j')}$.
 \item For any $i<i'\in I$, we have that $Q(i+1)= Q(i)+1<Q(i')$, and we obtain by \eqref{ecart3}
 \begin{align*}
 l_{i+1} + \Delta(Q(i+1),i+1) - (l_{i'}+\Delta(Q(i'),i'))&\geq \beta(i+1,i')+\Delta(Q(i+1),Q(i'))\\
 &\geq \chi(c_{Q(i+1)}\leq c_{Q(i')})\,\cdot
 \end{align*}
 By \eqref{odd}, we have that $\nu''_{Q(i)+1}>_c \nu''_{Q(i')}$.
\end{enumerate}
\item Parts coming from different kinds are well-ordered. In fact, by \Prp{pr2} and \Prp{pr21}, we can see that, for any $(j,i)\in J\times I$ such that $j>i$, we have
\begin{enumerate}
 \item $\psi(j,i)\leq 0 \Longleftrightarrow Q(i)+1<Q(j) \Longrightarrow \nu''_{Q(i)+1}>_c \nu''_{Q(j)}$,
 \item $\psi(j,i) > 0 \Longleftrightarrow Q(j)<Q(i) \Longrightarrow \nu''_{Q(j)}\gg_c \nu''_{Q(i)} $,
\end{enumerate}
Let us now consider the case $j<i$. We necessarily have that
$Q(j)\leq j<i\leq Q(i)$ so that $Q(j)<Q(i)$. We then have
\begin{align*}
\nu''_{Q(j)}-\nu''_{Q(i)} &= l_j-l_i - \Delta(j,i) + \Delta(Q(j),Q(i))\\
&\geq 1+ l'_j-l'_i - \Delta(j,i) + \Delta(Q(j),Q(i))\quad (\text{since }l'_i=2l_i\geq 1+l_i) \\ 
&\geq 1+  \Delta(Q(j),Q(i)) \quad \text{(by \eqref{ecart2})}\\
&\geq 1 + \chi(c_{Q(j)}\leq c_{Q(i)})
\end{align*}
We thus obtain  $\nu''_{Q(j)}\gg_c\nu''_{Q(i)}$. 
\end{enumerate}
To conclude, we observe that the final position is such that
\begin{itemize}
 \item For any $i\in I\,,\,\nu''_{Q(i)}= \nu''_{Q(i)+1}+\Delta(Q(i),Q(i)+1)\,,$
 \item for any $i\in (I+1)\cap\ssss\,\,,\nu''_{Q(i)}>_c\nu''_{Q(i)+1}\,,$
  \item for any $j\in J\cap\ssss\,\,,\nu''_{Q(j)}\gg_c\nu''_{Q(j)+1}\,,$
\end{itemize}
so that the final result $\nu'' =\nu''_1+\cdots+\nu''_s$ belongs to $\C$. Moreover, the set $Q(I)$ is exactly what we obtain for the set of indices of upper halves in \So by applying $\Phi$ on $\nu''$.

\end{proof}
We conclude the proof of \Prp{prpsi} by gathering \Prp{pr21} and \Prp{pr22}. 

\subsection{Reciprocity between $\Phi$ and $\Psi$}\label{recip}
In this section, we will show that $\Psi\circ\Phi = Id_{\C}$ and $\Phi\circ\Psi = Id_{\D}$.
\begin{enumerate}
\item Let us consider $\la\in \C$ and $\Phi(\la)=\la'_1+\cdots+\la'_{s}\in \D$ as in \Sct{prfphi}. The sets $P(I),P(J)$ are exactly what we obtained for indices of upper halves and parts that remain in $\Od'$.
By \Lem{lem1}, we then have for any $(j,i)\in J\times I$
\begin{align*}
\psi(P(j),P(i))& = (l_j+\Delta(P(j),j))-(l_i+\Delta(P(i),i))-\Delta(P(j),P(i))\\
&=l_j-l_i -\Delta(j,i)\,\cdot
\end{align*}
We thus conclude by \eqref{ecart} that
\begin{itemize}
\item for $j<i$, $\psi(P(j),P(i))\geq \alpha(j,i)>0$,
\item for $j>i$, $\psi(P(j),P(i))\leq \alpha(j,i)\leq 0$.
\end{itemize}
By \Prp{pr21} and \Prp{pr11}, we  then have  for any $i\in I$
\begin{align*}
Q(P(i))-P(i)&= |\{j\in J/\, P(j)>P(i), \psi(P(j),P(i))>0\}|\\
&= |\{j\in J/\, P(j)>P(i), j<i\}|\\
&= |\{j\in J/\, j<i,\phi(j,i)<0\}|\\
&= i-P(i)\,\,
\end{align*}
so that $Q(P(i))=i$, and then $Q(P(I))=I$ and $Q(P(I+1))=I+1$. By Remark \ref{rem1}, we also have $Q(P(J))=J$. This means by \Lem{lem21} and \Lem{lem1} that $\Psi(\Phi(\la))=\la$.
\item Let us now consider $\Psi(\nu)  = \nu''_1+\cdots+\nu''_s \in \C$ as in \Sct{prfpsi}. We saw  that after \So of $\Phi$, we exactly obtain $Q(I)$ as the set of indices of upper halves, $Q(I)+1$ the set of indices of lower halves, and 
$Q(J)$ the set of indices of parts that stay in $\Od'$.
By evaluating the corresponding function $\alpha$, we have for $Q(k)\leq Q(k') \in \sss$ that
\[\alpha(Q(k),Q(k')) = |[Q(k),Q(k'))\cap Q(J)|\,\text{ and }\alpha(Q(k'),Q(k))=-\alpha(Q(k),Q(k'))\,\cdot\]
Then for any $j,i\in J\times I$, we have by \Cr
\begin{align*}
\phi(Q(j),Q(i)) &= (l_j+\Delta(Q(j),j))-2(l_{i+1}-\Delta(i+1,Q(i)+1))-\Delta(Q(j),Q(i)+1)\\
&\quad-\Delta(Q(i)+1-\alpha(Q(j),Q(i)),Q(i)+1)\\
&=l_j-2l_{i+1}-[\Delta(j,Q(j)) + \Delta(Q(i)+1,i+1)+\Delta(Q(j),Q(i)+1)]\\
&\quad- [\Delta(Q(i)+1,i+1)+\Delta(Q(i)+1-\alpha(Q(j),Q(i)),Q(i)+1)]\\
& = l_j-2l_{i+1}-\Delta(j,i+1) -\Delta(Q(i)+1-\alpha(Q(j),Q(i)),i+1)\,\cdot
\end{align*}
But by computing $\alpha(Q(j),Q(i))$, since $Q$ is increasing on $J$ and $I\sqcup I+1$, we obtain
\begin{align*}
\alpha(Q(j),Q(i)) &=\alpha(1,Q(i))-\alpha(1,Q(j))\\
&= |[1,Q(i))\cap Q(J)|-|[1,Q(j))\cap Q(J)|\\
&= Q(i)-1 - |[1,Q(i))\cap Q(I\sqcup I+1)|-|[1,Q(j))\cap Q(J)|\\
&= Q(i)-1 - |[1,i)\cap (I\sqcup I+1)|-|[1,j)\cap J|\\
&= Q(i)-i + |[1,i)\cap J|-|[1,j)\cap J|\,\cdot
\end{align*}
We then have 
\begin{align}
 \phi(Q(j),Q(i)) &= l_j-2l_{i+1}-\Delta(j,i+1) -\Delta(i+1-|[1,i)\cap J|+|[1,j)\cap J|,i+1)\label{dif2}\\
 &= l_j-2l_i-\Delta(j,i) -\Delta(i+1-|[1,i)\cap J|+|[1,j)\cap J|,i)\label{dif3}\,\cdot
\end{align}
By \eqref{ecart2}, this gives :
\begin{enumerate}
 \item for $j<i$, by using \eqref{dif3}
 \begin{align*}
 \phi(Q(j),Q(i)) &\geq  \beta(j,i)-\Delta(i+1-|[1,i)\cap J|+|[1,j)\cap J|,i)\\
 &\geq \beta(j,i)-(|[1,i)\cap J|-|[1,j)\cap J|-1)\quad\text{(by \eqref{signe})}\\
 &= |(j,i]\cap J|-(|[j,i)\cap J|-1)\quad\text{(by \eqref{difff})}\\
 &= |(j,i)\cap J|-|(j,i)\cap J|\quad(\text{since }j\in J\text{ and } i\notin J)
 \end{align*}
  so that $\phi(Q(j),Q(i))\geq 0$,
\item  for $j>i$, we have $j>i+1$ and by using \eqref{dif2}
\begin{align*}
\phi(Q(j),Q(i)) &\leq -\beta(i+1,j)-\Delta(i+1-|[1,i)\cap J|+|[1,j)\cap J|,i+1)\\
&\leq -\beta(i+1,j)+|[i,j)\cap J|\quad\text{(by \eqref{signe})}\\
&= -|(i+1,j]\cap J|+|[i,j)\cap J|\quad\text{(by \eqref{difff})}\\
&= -1-|(i+1,j)\cap J|+|(i+1,j)\cap J|\quad(\text{since }j\in J\text{ and } i,i+1\notin J).
\end{align*}
so that 
$\phi(Q(j),Q(i))\leq -1 <0$ for any $j>i$.
\end{enumerate}
With this, by \Prp{pr11} and \Prp{pr21}, we obtain that the final position $P$ after \St of $\Phi$ is exactly such that, for any 
$i\in I$,  
\begin{align*}
Q(i)-P(Q(i))&= |\{j\in J/\, Q(j)<Q(i), \phi(Q(j),Q(i))<0\}|\\
&= |\{j\in J/\, Q(j)<Q(i), j>i\}|\\
&= |\{j\in J/\, j>i,\psi(j,i)>0\}|\\
&= Q(i)-i\,\,,
\end{align*}
so that $P(Q(i))=i$. As before, we obtain that $P(Q(X))=X$ for any $X\in \{I,I+1,J\}$. We  then have by \Lem{lem1} that $\Phi(\Psi(\nu))=\nu$.
\end{enumerate}

\section*{Acknowlegdements}
We  would like to thank Jeremy Lovejoy and Jehanne Dousse for advice given during the writing of the paper.


\begin{thebibliography}{999}

\bibitem{Alladi99}
K. ALLADI, \emph{A variation on a theme of Sylvester - a smoother road to G\"{o}llnitz's (Big) theorem}, Discrete Math. \textbf{196} (1999), 1--11.

\bibitem{AG93}
K. ALLADI and B. GORDON, \emph{Generalization of Schur's partition theorem}, Manuscripta Math. \textbf{79} (1993), 113--126.

\bibitem{BR80}
D. BRESSOUD, \emph{A combinatorial proof of Schur's 1926 partition theorem}, Proc. Amer. Math. Soc. \textbf{79} (1980), 338--340.

\bibitem{C93}
S. CAPPARELLI, \emph{On some representations of twisted affine Lie algebras
and combinatorial identities}, J. Algebra \textbf{154} (1993), 335--355.

\bibitem{CL06}
S. CORTEEL and J. LOVEJOY, \emph{An iterative-bijective approach to generalizations of Schur's theorem}, Eur. J. Comb. \textbf{27} (2006), 496--512.

\bibitem{Dousse16}
J. DOUSSE, \emph{Siladi\'c's theorem: weighted words, refinement and companion}, Proc. Amer. Math. Soc. \textbf{145} (2017), 1997--2009.

\bibitem{LW84}
J. LEPOWSKY and R. WILSON, \emph{The structure of standard modules, I: Universal algebras and the Rogers-Ramanujan identities}, Invent. Math. \textbf{77} (1984), 199--290.

\bibitem{MP87}
A. MEURMAN and M. PRIMC, \emph{Annihilating ideals of standard modules of
$sl(2,\mathbb{C})^\sim$ and combinatorial identities}, Adv. Math.\textbf{64} (1987), 177--240.

\bibitem{MP99}
A. MEURMAN and M. PRIMC, \emph{Annihilating ideals of standard modules of
$sl(2,\mathbb{C})^\sim$ and combinatorial identities}, Mem. Amer. Math. Soc. \textbf{137} (1999), viii + 89 pp.

\bibitem{RR19}
L. J. ROGERS and S. RAMANUJAN, \emph{Proof of certain identities in combinatory analysis}, Cambr. Phil. Soc. Proc. {\bf 19} (1919), 211-216.

\bibitem{PRC04}
PADMAVATHAMMA, R. RAGHAVENDRA and B. M. CHANDRASHEKARA, \emph{A new bijective proof of a partition theorem of K. Alladi}, Discrete Math. \textbf{287} (2004), 125--128.

\bibitem{Sc26}
I. SCHUR, \emph{Zur additiven zahlentheorie}, Sitzungsberichte der Preussischen Akademie der Wissenschaften (1926) , 488--495.

\bibitem{Si02}
I. SILADI\'C, \emph{Twisted $sl(3,\mathbb{C})^\sim$-modules and combinatorial identities}, Glas. Mat. Ser. III \textbf{52(72)} (2017), 53-77.


\end{thebibliography}
\end{document}